\documentclass[12pt,leqno]{amsart}
\usepackage{amsmath,amssymb,amsfonts}
\usepackage[all]{xy}

\usepackage[T5,T1]{fontenc}
\usepackage{mathpazo}
\usepackage{hyperref}
\usepackage{a4wide}

\theoremstyle{plain}
\newtheorem{thm}{Theorem}[section]
\newtheorem{lem}[thm]{Lemma}

\newtheorem{prop}[thm]{Proposition}

\theoremstyle{definition}

\newtheorem{ex}[thm]{Example}

\newtheorem{rmk}[thm]{Remark}

\newcommand{\surj}{\twoheadrightarrow}

\newcommand{\Gal}{{\rm Gal}}

\newcommand{\sC}{{\mathcal C}}

\newcommand{\F}{{\mathbb F}}

\newcommand{\Q}{{\mathbb Q}}

\newcommand{\U}{{\mathbb U}}

\newcommand{\Z}{{\mathbb Z}}

\newcommand{\res}{{\text{\sf res}\hspace{.1ex} }}

\newcommand*{\QEDB}{\hfill\ensuremath{\square}}
\def\NDT{{\fontencoding{T5}\selectfont Nguy\~ \ecircumflex n Duy T\^an}}

\begin{document}
\title{Construction of unipotent Galois extensions and Massey products }
\begin{abstract} 
For all primes $p$ and for all fields, we find a sufficient and necessary condition of the existence of a unipotent Galois extension of degree $p^6$. The main goal of this paper is to describe an explicit construction of such a Galois extension over fields  admitting such a Galois extension.
This construction is surprising in its simplicity and generality.  
The problem of finding such a construction has been left open since 2003. Recently a possible solution of this problem gained urgency because of an effort to extend   new advances in Galois theory and its relations with Massey products in Galois cohomology.
\end{abstract} 
\dedicatory{Dedicated to Alexander Merkurjev}
 \author{ J\'an Min\'a\v{c} and \NDT}
\address{Department of Mathematics, Western University, London, Ontario, Canada N6A 5B7}
\email{minac@uwo.ca}
\address{Department of Mathematics, Western University, London, Ontario, Canada N6A 5B7 and Institute of Mathematics, Vietnam Academy of Science and Technology, 18 Hoang Quoc Viet, 10307, Hanoi - Vietnam } 
\email{duytan@math.ac.vn}
\thanks{JM is partially supported  by the Natural Sciences and Engineering Research Council of Canada (NSERC) grant R0370A01. NDT is partially supported  by the National Foundation for Science and Technology Development (NAFOSTED) grant 101.04-2014.34}

\maketitle
\section{Introduction}
From the very beginning of the invention of Galois theory, one problem has emerged. For a given finite group $G$, find a Galois extension $K/\Q$ such that ${\rm Gal}(K/\Q)\simeq G$. This is still an open problem in  spite of the great efforts of a number of mathematicians and substantial progress having been made with specific groups $G$. (See \cite{Se3}.)
A more general problem is to ask the same question over other base fields $F$. This is a challenging and difficult problem even for groups G of prime power order.

In this paper we make  progress on this classical problem in Galois theory. Moreover this progress fits together well with a new development relating Massey products in Galois cohomology to basic problems in Galois theory. 
For all primes $p$ and all fields in the key test case of $n=4$, we construct Galois extensions with the unipotent Galois group $\U_n(\F_p)$ assuming only the existence of some Galois extensions of order $p^3$. 
  This fits into a program outlined in \cite{MT1} and \cite{MT2}, for the systematic construction of Galois $p$-closed extensions of general fields, assuming only knowledge of Galois extensions of degree less than or equal to $p^3$ and the structure of $p$th power roots of unity in the base field.
Thus both the methods and the  results in this paper pave the way to a program for obtaining  the structure of maximal pro-$p$-quotients of absolute Galois groups for all fields. We shall now describe some previous work of a number of mathematicians which has influenced our work, as well as its significance for further developments and applications.

Recently  there has been   substantial progress in Galois cohomology which has changed our perspective on Galois $p$-extensions over general fields. In some remarkable work, M. Rost and V. Voevodsky proved the Bloch-Kato conjecture on the structure of  Galois cohomology of general fields. (See \cite{Voe1,Voe2}.) 
From this work it follows that there must be enough Galois extensions to make higher degree Galois cohomology decomposable. However the explicit construction of such Galois extensions is completely mysterious. 
In \cite{MT1}, \cite{MT2} and \cite{MT5},  two new conjectures, the Vanishing $n$-Massey Conjecture and the Kernel $n$-Unipotent Conjecture were proposed. 
These conjectures in \cite{MT1} and \cite{MT2}, and the results in this paper, lead to a program of constructing these previously mysterious Galois extensions in a systematic way.
In these papers it is shown that the truth of these conjectures has some significant implications on the structure of absolute Galois groups.
These conjectures are based on a number of previous considerations. One motivation comes from topological considerations. (See \cite{DGMS} and \cite{HW}.) Another motivation is a program to describe various $n$-central series of absolute Galois groups as kernels of simple Galois representations. (See \cite{CEM, Ef, EM1,EM2, MSp,NQD,Vi}.)
If the Vanishing $n$-Massey Conjecture is true, then by a result in \cite{Dwy}, we obtain a program of building up $n$-unipotent Galois representations of absolute Galois groups  by induction on $n$. This is an attractive program because we  obtain a procedure of constructing larger Galois $p$-extensions from smaller ones, efficiently  using  the fact that certain {\it a priori} natural cohomological obstructions to this procedure  always vanish.

Recall that for each natural number $n$, $\U_n(\F_p)$ is the group of upper triangular $n\times n$-matrices with entries in $\F_p$ and diagonal entries 1. Then $\U_3(\F_2)$ is isomorphic to the  dihedral group of order 8, and if $p$ is odd, then $\U_3(\F_p)$ is isomorphic to the Heisenberg group $H_{p^3}$ of order $p^3$. For all $n\geq 4$ and all primes $p$, we can think of $\U_n(\F_p)$ as "higher Heisenberg groups" of order $p^{n(n-1)/2}$.
 It is now recognized that these groups play a very special role in  current Galois theory. Because $\U_n(\F_p)$ is a Sylow $p$-subgroup of ${\rm GL}_n(\F_p)$, and every finite $p$-group has a faithful linear $n$-dimensional representation   over $\F_p$, for some $n$, we see that every finite $p$-group can be embedded into $\U_n(\F_p)$ for some $n$. Besides, the Vanishing $n$-Massey Conjecture and the Kernel $n$-Unipotent Conjecture also indicate some deeper reasons why $\U_n(\F_p)$ is of special interest. 
 The constructions of Galois extensions with the Galois group $\U_3(\F_p)$ over fields which admit them, are well-known in the case when the base field is of characteristic not $p$. They are an important basic  tool in the Galois theory of $p$-extensions. (See for example \cite[Sections 6.5 and 6.6]{JLY}. Some early papers related to these topics like 
\cite{MNg} and \cite{M} now belong to classical work on Galois theory.)
 
  In \cite[Section 4]{GLMS},  a construction of Galois extensions $K/F$, ${\rm char}(F)\not=2$, with ${\rm Gal}(K/F)\simeq \U_4(\F_2)$, was discovered. Already at that time, one reason for searching for this construction was the motivation to find ideas to extend deep results on the characterization of  the fixed field  of the third 2-Zassenhaus filtration of an absolute Galois group $G_F$ as the compositum of Galois extensions of degree at most 8 (see \cite{Ef, EM2, MSp,Vi}), to a similar characterization of the fixed field of the fourth 2-Zassenhaus filtration of $G_F$. 
 In retrospect, looking at this construction, one recognizes some elements of the basic theory of Massey products. However at that time the authors of \cite{GLMS} were not familiar with Massey products. It was realized that such a construction would also be desirable  for $\U_4(\F_p)$ for all $p$ rather than $\U_4(\F_2)$, but none has been found until now. 
 
 In \cite{GLMS}, in the construction of a Galois field extension $K/F$ with ${\rm Gal}(K/F)\simeq \U_4(\F_2)$, a simple criteria    was used  for an element in $F$ to be a norm  from a bicyclic  extension of degree 4 modulo non-zero squares in the base field $F$.  However  in \cite{Me},  A. Merkurjev showed that a straightforward generalization of this criteria for $p$ odd instead of $p=2$, is not true in general. 
Therefore it was not clear whether such an \mbox{analogous} construction of Galois extensions $K/F$ with ${\rm Gal}(K/F)\simeq \U_4(\F_p)$ was possible for $p$ odd.
 
 On the other hand, a new consideration in \cite{HW}, \cite{MT1} and \cite{MT2} led us to formulate the Vanishing $n$-Massey Conjecture, and the most natural way to prove this conjecture for $n=3$ in the key non-degenerate case would be through constructing explicit Galois $\U_4(\F_p)$-extensions. 
 In fact we pursued both cohomological variants of proving the Vanishing 3-Massey Conjecture and the Galois theoretic construction of Galois $\U_4(\F_p)$-extensions. 
 
 The story of proving this conjecture and finally constructing Galois $\U_4(\F_p)$-extensions over all fields which admit them, is interesting. 
 First M. J. Hopkins and K. G. \mbox{Wickelgren} in \cite{HW} proved a result which implies that the Vanishing 3-Massey Conjecture with respect to prime 2, is true for all global fields of characteristic not 2. In \cite{MT1} we proved that the result of \cite{HW} is valid for any field $F$. 
 At the same time, in \cite{MT1} the Vanishing $n$-Massey Conjecture was formulated, and applications on the structure of the  quotients of absolute Galois groups were deduced.  
In \cite{MT3} we proved that the Vanishing 3-Massey Conjecture with respect to any prime $p$ is true for any global field $F$ containing a primitive $p$-th root of unity. 
In \cite{EMa}, I. Efrat and E. Matzri provided alternative proofs for the above-mentioned  results in \cite{MT1} and \cite{MT3}.
In \cite{Ma}, E. Matzri  proved that for any prime $p$ and for any field $F$ containing a primitive $p$-th root of unity,  every defined triple Massey product contains 0. This established the Vanishing 3-Massey Conjecture in the form formulated in \cite{MT1}.  Shortly after \cite{Ma} appeared on the arXiv,  two new preprints, \cite{EMa2} and \cite{MT5}, appeared nearly simultaneously and independently on the arXiv as well.  
In \cite{EMa2}, I. Efrat and E. Matzri replace \cite{Ma} and provide a cohomological approach to the proof of the main result in \cite{Ma}. In \cite{MT5} we also provide a cohomological method of proving the same result. We also extend the vanishing of triple Massey products to all fields, and thus remove the restriction that the base field contains a primitive $p$-th root of unity. 
We further provide applications on the structure of  some canonical quotients of absolute Galois groups, and also  show that some special higher $n$-fold Massey products vanish. Finally in this paper we are able to provide a construction of the Galois $\U_4(\F_p)$-extension $M/F$ for any field $F$  which admits such an extension. 
We use this construction to provide a natural new proof, which we were seeking from the beginning of our search for a Galois theoretic proof, of the vanishing of triple Massey products over all fields. 

Some interesting cases of  "automatic" realizations of Galois groups  are known. These are  cases  when the existence of one Galois group over a given field forces the existence of some other Galois groups over this field.  (See for example \cite{Je, MS,MSS, MZ, Wh}.) However, nontrivial cases of automatic realizations coming from an actual construction of embedding  smaller Galois extensions to larger ones, are relatively rare, and they are difficult to produce. 
In our construction we are able,  from  knowledge of the existence of two Heisenberg Galois extensions of degree $p^3$ over a given base field $F$ as above, to find possibly another pair of Heisenberg Galois extensions whose compositum can be automatically embedded in a Galois $\U_4(\F_p)$-extension. (See also Remark~\ref{rmk:modification}.)
Observe that in all proofs of the Vanishing 3-Massey Conjecture we currently have, constructing Heisenberg Galois extensions of degree $p^3$ has  played an important role. For the sake of a possible inductive proof of the Vanishing $n$-Massey Conjecture, it seems  important to be able to  inductively construct Galois $\U_n(\F_p)$-extensions. This  has now been achieved for the induction step from $n=3$ to $n=4$, and it opens up a way to approach the Vanishing 4-Massey Conjecture. 

Another motivation for this work which combines well with the motivation described above, comes from anabelian birational considerations. Very roughly in various generality and precision, it was observed that small canonical quotients of absolute Galois groups determine surprisingly precise information about  base fields, in some cases  entire base fields up to isomorphisms. (See \cite{BT1,BT2,CEM,EM1,EM2,MSp,Pop}.) 
But these results suggest that some small canonical quotients of an absolute Galois group together with  knowledge of the roots of unity in the base field should determine larger canonical quotients of this absolute Galois group.
The Vanishing $n$-Massey Conjecture and the  Kernel $n$-Unipotent Conjecture, together with the program of explicit constructions of Galois $\U_n(\F_p)$-extensions, make this project more precise. Thus our main results, Theorems~\ref{thm:construction}, \ref{thm:construction char not p}, \ref{thm:construction char p} and~\ref{thm:U4}, contribute to this project.

A further potentially important application for this work is the theory of Galois $p$-extensions of global fields with restricted ramification and questions surrounding the Fontaine-Mazur conjecture. (See \cite{Ko}, \cite{La}, 
\cite{McL}, \cite{Ga},\cite{Se2}.) For example in \cite[Section 3]{McL}, there is a criterion for infinite Hilbert $p$-class field towers over quadratic imaginary number fields relying on the vanishing of certain triple Massey products. The explicit constructions in this paper should be useful for approaching  these classical number theoretic problems. 

Only relatively recently, the  investigations of the Galois realizability of some larger $p$-groups among families of small $p$-groups, appeared. (See the very interesting papers \cite{Mi1}, \cite{Mi2}, \cite{GS}.)
 In these papers the main concern is understanding cohomological and Brauer group obstructions for the realizability of Galois field extensions with prescribed Galois groups. In our paper the main concern is the explicit constructions and their connections with Massey products.
 In other recent papers \cite{CMS} and \cite{Sch}, the authors succeeded to treat the cases of characteristic equal to $p$ or not equal to $p$, nearly uniformly. This is also the case with  our paper.    

Our paper is organized as follows. In Section 2 we recall basic notions about norm residue symbols and Heisenberg extensions of degree $p^3$. (For convenience we think of  the dihedral group of order 8 as the Heisenberg group of order 8.) In Section 3 we provide a detailed construction of Galois $\U_4(\F_p)$-extensions beginning with two "compatible" Heisenberg extensions of degree $p^3$. Section 3 is divided into two subsections. In Subsection 3.1 we provide a construction of the required Galois extension $M/F$ over any field $F$ which contains a primitive $p$-th root of unity. In Subsection 3.2 we provide such a construction
for all fields of characteristic not $p$, building on the results and methods in Subsection 3.1. 
In Example~\ref{ex:p=2} we illustrate our method on a surprisingly simple construction of Galois $\U_4(\F_2)$-extensions over any field $F$ with ${\rm char}(F)\not=2$. 
In Section 4 we provide a required construction for all fields of characteristic $p$.
After the original and classical papers of E. Artin and O. Schreier \cite{ASch} and E. Witt \cite{Wi}, these constructions seem to add new results on the construction of basic Galois extensions $M/F$ with Galois groups $\U_n(\mathbb{F}_p)$, $n = 3$ and $n = 4$.  These are aesthetically pleasing constructions with remarkable simplicity.  They follow constructions in characteristic not $p$, but they are simpler. See also \cite[Section 5.6 and Appendix A1]{JLY} for another procedure to obtain these Galois extensions. 
In Section 5 we provide a new natural Galois theoretic proof of the vanishing of triple Massey products over all fields in the key non-degenerate case. We also complete the new proof of the vanishing of triple Massey products in the case when a primitive $p$-th root of unity is contained in the base field.
Finally we formulate a necessary and sufficient condition for the existence of a Galois $\U_4(\F_p)$-extension $M/F$ which contains an elementary $p$-extension of any field $F$ (described by three linearly independent characters), and we summarize the main results in Theorem~\ref{thm:U4}.
\\
\\
{\bf Acknowledgements: } We would like to thank  M. Ataei, L. Bary-Soroker, S. K. Chebolu, I. Efrat, H. \'{E}snault, E. Frenkel, S. Gille, J. G\"{a}rtner, P. Guillot, D. Harbater, M. J. Hopkins, Ch. Kapulkin, I. K\v{r}\'i\v{z}, J. Labute, T.-Y. Lam, Ch. Maire, E. Matzri,   C. McLeman, D. Neftin, J. Nekov\'a\v{r}, R. Parimala, C. Quadrelli, M. Rogelstad, A. Schultz, R. Sujatha, Ng. Q. {\fontencoding{T5}\selectfont Th\'\abreve{}ng}, A. Topaz, K. G. Wickelgren and O. Wittenberg
for having been able to share our enthusiasm for this relatively new subject of Massey products in Galois cohomology, and for their encouragement, support, and  inspiring discussions.   We are  very grateful to the  anonymous referee for his/her  careful reading of our paper, and for providing us with insightful comments and valuable suggestions which we used to improve our exposition.
\\
\\
{\bf Notation:} If $G$ is a group and $x,y\in G$, then $[x,y]$ denotes the commutator $xy x^{-1}y^{-1}$. 
For any element $\sigma$ of  finite order $n$ in $G$, we denote $N_{\sigma}$ to be  the element $1 +\sigma +\cdots +\sigma^{n-1}$ in the integral group ring $\Z[G]$ of G.

For a field $F$, we denote $F_s$ (respectively $G_F$) to be its separable closure (respectively its absolute Galois group ${\rm Gal}(F_s/F))$. We denote $F^\times$ to be the set of non-zero elements of $F$.
 For a given profinite group $G$, we call a Galois extension $E/F$, a (Galois) {\it $G$-extension} if the Galois group ${\rm Gal}(E/F)$ is isomorphic to $G$.
 
For a unital commutative ring $R$ and  an integer $n\geq 2$, we denote $\U_n(R)$ as the group of all upper-triangular unipotent $n\times n$-matrices with entries in $R$.  
For any (continuous) representation $\rho\colon G \to \U_n(R)$ from a (profinite) group $G$ to $\U_n(R)$ (equipped with discrete topology ), and $1\leq i< j\leq n$, let $\rho_{ij}\colon G \to R$ be the composition of $\rho$  with the projection from $\U_n(R)$ to its $(i,j)$-coordinate.

\section{Heisenberg extensions}
\label{sec:Heisenberg}
The materials in this section have been taken from \cite[Section 3]{MT5}.
\subsection{Norm residue symbols}
\label{subsec:norm residue}
Let $F$ be a field containing a primitive $p$-th root of unity $\xi$. 
For any element $a$ in $F^\times$, we shall write $\chi_a$  for the  character corresponding to $a$ via the Kummer map $F^\times\to H^1(G_F,\Z/p\Z)={\rm Hom}(G_F,\Z/pZ)$.  From now on we assume that $a$ is not in $(F^\times)^p$. The extension $F(\sqrt[p]{a})/F$ is a Galois extension with the Galois group $\langle \sigma_a\rangle\simeq \Z/p\Z$, where $\sigma_a$ satisfies $\sigma_a(\sqrt[p]{a})=\xi\sqrt[p]{a}$. 

The character $\chi_a$ defines a homomorphism $\chi^a\in {\rm Hom}(G_F,\frac{1}p\Z/\Z)\subseteq {\rm Hom}(G_F,\Q/\Z)$ by the formula
\[
\chi^a =\frac{1}{p} \chi_a. 
\]
Let $b$ be any element in $F^\times$.  Then the norm residue symbol may be defined as
\[
(a,b):= (\chi^a,b):= b\cup \delta \chi^a.
\]
Here $\delta$ is the coboundary homomorphism $\delta \colon H^1(G,\Q/\Z)\to H^2(G,\Z)$ associated to the short exact sequence of trivial $G$-modules
\[
0\to \Z\to \Q\to \Q/\Z\to 0.
\]

The cup product $\chi_a\cup \chi_b\in H^2(G_F,\Z/p\Z)$ can be interpreted as the norm residue symbol $(a,b)$. More precisely, we consider the exact sequence
\[
0\longrightarrow \Z/p\Z \longrightarrow  F_s^\times \stackrel{x\mapsto x^p}{\longrightarrow} F_s^\times \longrightarrow 1,
\]
where $\Z/p\Z$ has been identified with the group of $p$-th roots of unity $\mu_p$ via the choice of $\xi$. As $H^1(G_F,F_s^\times)=0$, we obtain
\[
0{\longrightarrow} H^2(G_F,\Z/p\Z)\stackrel{i}{\longrightarrow} H^2(G_F,F_s^\times) \stackrel{\times p}{\longrightarrow} H^2(G_F,F_s^\times).
\]
Then one has $i(\chi_a\cup \chi_b)=(a,b)\in H^2(G_F,F_s^\times)$. (See \cite[Chapter XIV, Proposition 5]{Se}.) 
\subsection{Heisenberg extensions}
\label{subsec:Heisenberg}
In this subsection we recall some basic facts about Heisenberg extensions. (See \cite[Chapter 2, Section 2.4]{Sha} and \cite[Sections 6.5 and 6.6 ]{JLY}.)
 
Assume that $a,b$ are elements in $F^\times$, which are linearly independent modulo $(F^\times)^p$. Let $K= F(\sqrt[p]{a},\sqrt[p]{b})$. Then $K/F$ is a Galois extension whose Galois group is generated by $\sigma_a$ and $\sigma_b$. Here $\sigma_a(\sqrt[p]{b})=\sqrt[p]{b}$, $\sigma_a(\sqrt[p]{a})=\xi \sqrt[p]{a}$; $\sigma_b(\sqrt[p]{a})=\sqrt[p]{a}$, $\sigma_b(\sqrt[p]{b})=\xi \sqrt[p]{b}$. 

We consider a map $\U_3(\Z/p\Z)\to (\Z/p\Z)^2$ which sends $\begin{bmatrix} 
1 & x & z\\
0 & 1 & y\\
0 & 0 & 1 
\end{bmatrix}$
to $(x,y)$. Then we have the following embedding problem
\[
 \xymatrix{
& & &G_F \ar@{->}[d]^{\bar\rho} \\
0\ar[r]& \Z/p\Z \ar[r] &\U_3(\Z/p\Z)\ar[r] &(\Z/p\Z)^2\ar[r] &1,
}
\]
where $\bar\rho$ is the map $(\chi_a,\chi_b)\colon G_F\to {\rm Gal}(K/F)\simeq (\Z/p\Z)^2$. (The last isomorphism ${\rm Gal}(K/F)\simeq (\Z/p\Z)^2$ is the one which sends $\sigma_a$ to $(1,0)$ and $\sigma_b$ to $(0,1)$.)

Assume that $\chi_a\cup \chi_b=0$. Then the norm residue symbol $(a,b)$ is trivial. Hence there exists $\alpha$ in $F(\sqrt[p]{a})$ such that $N_{F(\sqrt[p]{a})/F}(\alpha)=b$ (see  \cite[Chapter XIV, Proposition 4 (iii)]{Se}). We set 
\[
A_0=\alpha^{p-1} \sigma_a(\alpha^{p-2})\cdots \sigma_a^{p-2}(\alpha)=\prod_{i=0}^{p-2} \sigma_a^{i}(\alpha^{p-i-1}) \in  F(\sqrt[p]{a}).
\]
\begin{lem} Let $f_a$ be an element in $F^\times$. Let $A=f_aA_0$. Then we have
\label{lem:operator}
\[
\frac{\sigma_a(A)}{A}=\frac{N_{F(\sqrt[p]{a})/F}(\alpha)}{\alpha^p}=\frac{b}{\alpha^p}.
\]
%\QEDB
\end{lem}
\begin{proof}Observe that $\dfrac{\sigma_a(A)}{A}=\dfrac{\sigma_a(A_0)}{A_0}$.
The lemma then follows from the identity
\[
(s-1)\sum_{i=0}^{p-2} (p-i-1)s^{i} = \sum_{i=0}^{p-1} s^i -p s^0.
\qedhere
\]
\end{proof}

\begin{prop}
\label{prop:Heisenberg extension}
Assume that $\chi_a\cup \chi_b=0$. Let $f_a$ be an element in $F^\times$. Let  $A=f_aA_0$ be defined as above. Then the homomorphism $\bar{\rho}:= (\chi_a,\chi_b)\colon G_F\to \Z/p\Z\times \Z/p\Z$ lifts to a Heisenberg extension $\rho\colon G_F\to \U_3(\Z/p\Z)$. 
\end{prop}
\begin{proof}[Sketch of Proof] 
Let $L:=K(\sqrt[p]{A})/F$. Then $L/F$ is Galois extension. 
Let $\tilde{\sigma}_a\in {\rm Gal}(L/F)$ (resp. $\tilde{\sigma}_b\in {\rm Gal}(L/F)$) be an extension of $\sigma_a$ (resp. $\sigma_b$). Since $\sigma_b(A)=A$, we have $\tilde{\sigma}_b(\sqrt[p]{A})=\xi^j\sqrt[p]{A}$, for some $j\in \Z$. 
Hence $\tilde{\sigma}_b^p(\sqrt[p]{A})=\sqrt[p]{A}$. This implies that $\tilde{\sigma}_b$ is of order $p$.

On the other hand, we have 
$
\tilde{\sigma}_a(\sqrt[p]{A})^p=\sigma_a(A) =A \dfrac{b}{\alpha^p}.
$
Hence $\tilde{\sigma}_a(\sqrt[p]{A})=\xi^i \sqrt[p]{A}\dfrac{\sqrt[p]{b}}{\alpha}$, for some $i\in \Z$. Then $\tilde{\sigma}_a^{p}(\sqrt[p]{A})=\sqrt[p]{A}$. Thus $\tilde{\sigma}_a$ is of order $p$. 

If we set $\sigma_A:=[\tilde{\sigma}_a,\tilde{\sigma}_b]$, then $\sigma_A(\sqrt[p]{A})=\xi^{-1} \sqrt[p]{A}$. This implies that $\sigma_A$ is of order $p$. Also one can check that
\[
[\tilde{\sigma}_a,\sigma_A]=[\tilde{\sigma}_b,\sigma_A]=1.
\]

We can define an isomorphism $\varphi \colon {\rm Gal}(L/F)\to \U_3(\Z/p\Z)$ by letting
\[
\sigma_a \mapsto \begin{bmatrix}
1& 1 & 0 \\
0& 1 & 0 \\
0& 0 & 1
\end{bmatrix},
\sigma_b\mapsto 
\begin{bmatrix}
1& 0 & 0 \\
0& 1 & 1 \\
0& 0 & 1 
\end{bmatrix},
\sigma_{A}\mapsto 
\begin{bmatrix}
1& 0 & 1 \\
0& 1 & 0 \\
0& 0 & 1 
\end{bmatrix}.
\]
Then the composition $\rho\colon G_F\to {\rm Gal}(L/F)\stackrel{\varphi}{\to} \U_3(\Z/p\Z)$ is the desired lifting of $\bar{\rho}$. 

Note that $[L:F]=p^3$. Hence there are exactly $p$ extensions of $\sigma_a\in {\rm Gal}(E/F)$ to the automorphisms in ${\rm Gal}(L/F)$ since $[L:E]=p^3/p^2=p$. Therefore for later use,  we can choose an extension, still denoted by $\sigma_a\in {\rm Gal}(L/F)$, of $\sigma_a\in {\rm Gal}(K/F)$ in such a way that $\sigma_a(\sqrt[p]{A})= \sqrt[p]{A}\dfrac{\sqrt[p]{b}}{\alpha}$. 
\end{proof}
%%%%%%%%%%%%%%%%%%%%%%%%%%%%%%%%%%%%%%%%%%
%%%%%%%%%%%%%%%%%%%%%%%%%%%%%%%%%%%%%%%%%%%
\section{The construction of $\U_4(\F_p)$-extensions: the case of characteristic $\not=p$}
\subsection{Fields containing primitive $p$-th roots of unity}
\label{subsec:with p-primitive}
In this subsection we assume that $F$ is a field containing a primitive $p$-th root $\xi$ of unity.   
The following result can be deduced from Theorem~\ref{thm:U4}, but for the convenience of the reader we include a proof here.
\begin{prop}
\label{prop:U4 existence}
 Assume that there exists a Galois extension $M/F$ such that ${\rm Gal}(M/F)\simeq \U_4(\F_p)$. Then there exist $a,b,c\in F^\times$ such that $a,b,c$ are  linearly independent modulo $(F^\times)^p$ and $(a,b)=(b,c)=0$. Moreover $M$ contains $F(\sqrt[p]{a},\sqrt[p]{b},\sqrt[p]{c})$. 
\end{prop}
\begin{proof}
 Let $\rho$ be the composite $\rho\colon G_F\surj{\rm Gal}(M/F)\simeq \U_4(\F_p)$. Then $\rho_{12},\rho_{23}$ and $\rho_{34}$ are elements in ${\rm Hom}(G_F,\F_p)$. Hence there are $a,b$ and $c$ in $F^\times$ such that $\chi_a=\rho_{12}$, $\chi_b=\rho_{23}$ and $\chi_c=\rho_{34}$. Since $\rho$ is a group homomorphism, by looking at the coboundaries of $\rho_{13}$ and $\rho_{24}$, we see that 
\[
\chi_a\cup \chi_b =\chi_b\cup\chi_c=0 \in H^2(G_F,\F_p).
\]
 This implies that $(a,b)=(b,c)=0$ by \cite[Chapter XIV, Proposition 5]{Se}.
 
 Let $\varphi:=(\chi_a,\chi_b,\chi_c)\colon G_F\to (\F_p)^3$. Then $\varphi$ is surjective. By Galois correspondence, we  have 
 \[
 {\rm Gal}(F_s/F(\sqrt[p]{a},\sqrt[p]{b},\sqrt[p]{c}))= \ker\chi_a\cap \ker\chi_b\cap\ker\chi_c=\ker\varphi.
 \]
 This implies that ${\rm Gal}(F(\sqrt[p]{a},\sqrt[p]{b},\sqrt[p]{c})/F)\simeq (\F_p)^3$. Hence by Kummer theory, we see that $a,b$ and $c$ are  linearly independent modulo $(F^\times)^p$. Clearly, $M$ contains $F(\sqrt[p]{a},\sqrt[p]{b},\sqrt[p]{c})$. 
\end{proof}

Conversely we shall see in this section that given these necessary conditions  for the existence of $\U_4(\F_p)$-Galois extensions over  $F$, as in Proposition~\ref{prop:U4 existence}, we can construct a Galois extension $M/F$ with the Galois group isomorphic to $\U_4(\F_p)$. 

From now on we assume that we are given  elements $a$, $b$ and $c$ in $F^\times$ such that 
 $a$, $b$ and $c$ are linearly independent modulo $(F^\times)^p$ and that 
  $(a,b)=(b,c)=0$. We shall construct a Galois $\U_4(\F_p)$-extension $M/F$ such that $M$ contains $F(\sqrt[p]{a},\sqrt[p]{b},\sqrt[p]{c})$.

First we note that  $F(\sqrt[p]{a},\sqrt[p]{b},\sqrt[p]{c})/F$ is a Galois extension with ${\rm Gal}(F(\sqrt[p]{a},\sqrt[p]{b},\sqrt[p]{c})/F)$ generated by $\sigma_a,\sigma_b,\sigma_c$. Here
\[
\begin{aligned}
\sigma_a(\sqrt[p]{a})&=\xi \sqrt[p]{a}, \sigma_a(\sqrt[p]{b})=\sqrt[p]{b}, \sigma_a(\sqrt[p]{c})=\sqrt[p]{c};\\
\sigma_b(\sqrt[p]{a})&=\sqrt[p]{a}, \sigma_b(\sqrt[p]{b})=\xi \sqrt[p]{b}, \sigma_b(\sqrt[p]{c})= \sqrt[p]{c};\\
\sigma_c(\sqrt[p]{a})&=\sqrt[p]{a}, \sigma_c(\sqrt[p]{b})=\sqrt[p]{b}, \sigma_c(\sqrt[p]{c})=\xi \sqrt[p]{c}.
\end{aligned}
\]

Let $E= F(\sqrt[p]{a},\sqrt[p]{c})$. Since $(a,b)=(b,c)=0$, there are $\alpha$ in $F(\sqrt[p]{a})$ and $\gamma$ in $F(\sqrt[p]{c})$ (see  \cite[Chapter XIV, Proposition 4 (iii)]{Se}) such that
\[
N_{F(\sqrt[p]{a})/F}(\alpha)=b=N_{F(\sqrt[p]{c})/F}(\gamma).
\]

Let $G$ be the Galois group ${\rm Gal}(E/F)$. Then $G=\langle \sigma_a,\sigma_c \rangle $, where $\sigma_a\in G$ (respectively $\sigma_c\in G$) is the restriction of $\sigma_a\in {\rm Gal}(F(\sqrt[p]{a},\sqrt[p]{b},\sqrt[p]{c})/F)$ (respectively $\sigma_c\in {\rm Gal}(F(\sqrt[p]{a},\sqrt[p]{b},\sqrt[p]{c})/F)$).

Our next goal is to find an element $\delta$ in $E^\times$ such that the Galois closure of  $E(\sqrt[p]{\delta})$ is our desired $\U_4(\F_p)$-extension of $F$.
We define 
\[
\begin{aligned}
C_0=\prod_{i=0}^{p-2} \sigma_c^{i}(\gamma^{p-i-1}) \in  F(\sqrt[p]{a}),
\end{aligned}
\] 
and define $B:=\gamma/\alpha$. Then we have the following result, which follows from Lemma~\ref{lem:operator} (see \cite[Proposition 3.2]{Ma} and/or \cite[Lemma 4.2]{MT5}). 
\begin{lem}
\label{lem:operators}
 We have
\begin{enumerate}
\item $\dfrac{\sigma_a(A_0)}{A_0}=N_{\sigma_c}(B)$.
\item $\dfrac{\sigma_c(C_0)}{C_0}=N_{\sigma_a}(B)^{-1}$.
\QEDB
\end{enumerate}
\end{lem}

\begin{rmk}
\label{rmk:modification}
 We would like to  informally explain the meaning of the next lemma. 
  From our hypothesis $(a,b)=0=(b,c)$ and from Subsection~\ref{subsec:Heisenberg}, we see that we can obtain two Heisenberg extensions $L_1=F(\sqrt[p]{a},\sqrt[p]{b},\sqrt[p]{A_0})$ and $L_2=F(\sqrt[p]{b},\sqrt[p]{c},\sqrt[p]{C_0})$ of $F$. Here we have chosen specific elements $A_0\in F(\sqrt[p]{a})$ and $C_0\in F(\sqrt[p]{c})$. However we may  not be able to embed  the compositum of $L_1$ and $L_2$ into our desired Galois extension $M/F$ with $\Gal(M/F)\simeq \U_4(\F_p)$. 
  We know that we can modify the element $A_0$ by any element $f_a\in F^\times$ and the element $C_0$ by any element $f_c\in F^\times$ obtaining elements $A=f_aA_0$ and $C=f_cC_0$ instead of $A_0$ and $C_0$. This new choice of elements may change the fields $L_1$ and $L_2$ but the new fields will  still be Heisenberg extensions containing $F(\sqrt[p]{a},\sqrt[p]{b})$ and $F(\sqrt[p]{b},\sqrt[p]{c})$ respectively.
The next lemma will provide us with a suitable  modification of $A_0$ and $C_0$. From the proof of Theorem~\ref{thm:construction} we shall see that the compositum of these modified Heisenberg extensions can indeed be embedded into a Galois extension $M/F$ with $\Gal(M/F)\simeq \U_4(\F_p)$. This explains our comment in the introduction in the paragraph related to the "automatic realization of Galois groups".
\end{rmk}

\begin{lem}
\label{lem:modification}
Assume that there exist $C_1, C_2\in E^\times$ such that 
\[ B=\frac{\sigma_a(C_1)}{C_1} \frac{C_2}{\sigma_c(C_2)}.\]
Then $N_{\sigma_c}(C_1)/A_0$ and $N_{\sigma_a}(C_2)/C_0$ are in $F^\times$. Moreover, if we let $A=N_{\sigma_c}(C_1)\in F(\sqrt[p]{a})^\times$ and $C=N_{\sigma_a}(C_2)\in F(\sqrt[p]{c})^\times$, then  there exists 
 $\delta \in E^\times$ such that
\[
\begin{aligned}
\frac{\sigma_c(\delta)}{\delta}&= A C_1^{-p},\\
\frac{\sigma_a(\delta)}{\delta}&=C C_2^{-p}.
\end{aligned}
\]
\end{lem}

\begin{proof} By Lemma~\ref{lem:operators}, we have
\[
\frac{\sigma_a(A_0)}{A_0}= N_{\sigma_c}(B) = N_{\sigma_c}\left(\frac{\sigma_a(C_1)}{C_1}\right) N_{\sigma_c}\left(\frac{C_2}{\sigma_c(C_2)}\right)
= \frac{\sigma_a(N_{\sigma_c}(C_1))}{N_{\sigma_c}(C_1)}.
\]
This implies that
\[
\frac{N_{\sigma_c}(C_1)}{A_0}= \sigma_a\left(\frac{N_{\sigma_c}(C_1)}{A_0}\right).
\]
Hence
\[
\dfrac{N_{\sigma_c}(C_1)}{A_0} \in F(\sqrt[p]{c})^\times\cap F(\sqrt[p]{a})^\times=F^\times.
\]

By Lemma~\ref{lem:operators}, we have
\[
\frac{\sigma_c(C_0)}{C_0}= N_{\sigma_a}(B^{-1}) = N_{\sigma_a}\left(\frac{C_1}{\sigma_a(C_1)}\right) N_{\sigma_a}\left(\frac{\sigma_c(C_2)}{C_2}\right)
= \frac{\sigma_c(N_{\sigma_a}(C_2))}{N_{\sigma_a}(C_2)}.
\]
This implies that
\[
\frac{N_{\sigma_a}(C_2)}{C_0}= \sigma_c\left(\frac{N_{\sigma_a}(C_2)}{C_0}\right).
\]
Hence
\[
\dfrac{N_{\sigma_a}(C_2)}{C_0} \in F(\sqrt[p]{a})^\times\cap F(\sqrt[p]{c})^\times=F^\times.
\]

 Clearly, one  has
\[
\begin{aligned}
N_{\sigma_a}(CC_2^{-p})&=1,\\
N_{\sigma_c}(A C_1^{-p})&=1.
\end{aligned}
\]
We also have
\[
\begin{aligned}
\frac{\sigma_a(AC_1^{-p})}{AC_1^{-p}}
 \frac{CC_2^{-p}}{\sigma_c(CC_2^{-p})}&= \frac{\sigma_a(A)}{A} \left(\frac{\sigma_a(C_1)}{C_1}\right)^{-p} \frac{C}{\sigma_c(C)}\left(\frac{C_2}{\sigma_c(C_2)}\right)^{-p}\\
&= \frac{b}{\alpha^p}\frac{\gamma^p}{b} B^{-p}\\
&=1.
\end{aligned}
\]
Hence, we have
\[
\frac{\sigma_a(AC_1^{-p})}{AC_1^{-p}}= \frac{\sigma_c(CC_2^{-p})}{CC_2^{-p}}.
\]
From \cite[page 756]{Co} we see that there exists $\delta \in E^\times$ such that 
\[
\begin{aligned}
\frac{\sigma_c(\delta)}{\delta}&= A C_1^{-p},\\
\frac{\sigma_a(\delta)}{\delta}&=C C_2^{-p},
\end{aligned}
\]
as desired.
\end{proof}
\begin{rmk}
 The result of I. G. Connell which we use in the above proof, is a variant of Hilbert's Theorem 90. This result was independently discovered by S. Amitsur and D. Saltman in \cite[Lemma 2.4]{AS}. (See also  \cite[Theorem 2]{DMSS} for the case $p=2$.)
\end{rmk}
\begin{lem}
\label{lem:C1C2}
There exists $e\in E^\times$  such that $B= \dfrac{\sigma_a\sigma_c(e)}{e}$. Furthermore, for such an element $e$ the following statements are true.
\begin{enumerate}
\item If we set $C_1:=\sigma_c(e)\in E^\times$, $C_2:=e^{-1}\in E^\times$, then $B=\dfrac{\sigma_a(C_1)}{C_1} \dfrac{C_2}{\sigma_c(C_2)}$.
\item If we set $C_1:=e \in E^\times$, $C_2:=(eB)\sigma_c(eB)\cdots \sigma_c^{p-2}(eB)\in E^\times$, then $B=\dfrac{\sigma_a(C_1)}{C_1} \dfrac{C_2}{\sigma_c(C_2)}$.
\end{enumerate}
\end{lem}
\begin{proof}
We have 
\[ N_{\sigma_a\sigma_c}(B)=\frac{N_{\sigma_a\sigma_c}(\alpha)}{N_{\sigma_a\sigma_c}(\gamma)}=\frac{N_{\sigma_a}(\alpha)}{N_{\sigma_c}(\gamma)}=\frac{b}{b}=1.
\]
 Hence by Hilbert's Theorem 90, there exists $e \in E^\times$ such  that 
$B=\dfrac{\sigma_a\sigma_c(e)}{e}. $
 \\
 \\
 (1) Clearly, we have 
 \[
 \dfrac{\sigma_a(C_1)}{C_1} \frac{C_2}{\sigma_c(C_2)}=\dfrac{\sigma_a(\sigma_c(e))}{\sigma_c(e)} \frac{e^{-1}}{\sigma_c(e^{-1})}=\dfrac{\sigma_a\sigma_c(e)}{e}=B.
 \]
  (2) From $B=\dfrac{\sigma_a\sigma_c(e)}{e}$, we see that $eB= \sigma_a\sigma_c(e)$. Hence $\sigma_c^{p-1}(eB)=\sigma_a(e)$. Therefore 
  \[
  B=\frac{\sigma_a(e)}{e}\frac{eB}{\sigma_c^{p-1}(eB)}=  \frac{\sigma_a(C_1)}{C_1} \frac{C_2}{\sigma_c(C_2)}.
  \qedhere
  \]
\end{proof}
%%%%%%%%%%%%%%%%%%%%%%%%%%%%%%%%%%%%%%%%%%%%%
%%%%%%%%%%%%%%%%%%%%%%%%%%%%%%%%%%%%%%%%%%%%%%%
\begin{thm} 
\label{thm:construction}
Let the  notation and assumption be as in Lemma~\ref{lem:modification}. Let $M:= E(\sqrt[p]{\delta},\sqrt[p]{A},\sqrt[p]{C},\sqrt[p]{b})$. Then $M/F$ is a Galois extension, $M$ contains $F(\sqrt[p]{a},\sqrt[p]{b},\sqrt[p]{c})$, and ${\rm Gal}(M/F)\simeq \U_4(\F_p)$.
\end{thm}
\begin{proof}
Let $W^*$ be the $\F_p$-vector space in $E^\times/(E^\times)^p$ generated by $[b]_E,[A]_E,[C]_E$ and $[\delta]_E$. Here for any $0\not= x$ in a field $L$, we denote $[x]_L$ the image of $x$ in $L^\times/(L^\times)^p$. Since 
\begin{align}
\sigma_c(\delta)&= \delta A C_1^{-p} & \text{ (by Lemma~\ref{lem:modification})}, \label{eq:1}\\
\sigma_a(\delta)&=\delta C C_2^{-p} &\text{ (by Lemma~\ref{lem:modification})},\label{eq:2}\\
\sigma_a(A)&=A \frac{b}{\alpha^p} \label{eq:3} &\text{ (by Lemma~\ref{lem:operator})},\\
\sigma_c(C)&=C \frac{b}{\gamma^p} \label{eq:4} &\text{ (by Lemma~\ref{lem:operator})},
\end{align}
we see that $W^*$ is in fact an $\F_p[G]$-module. Hence $M/F$ is a Galois extension by Kummer theory.
\\
\\
{\bf Claim:} $\dim_{\F_p}(W^*)=4$. Hence $[L:F]=[L:E][E:F]=p^4p^2=p^6.$\\
{\it Proof of Claim:} From our hypothesis that $\dim_{\F_p}\langle [a]_F,[b]_F,[c]_F\rangle=3$, we see that $\langle [b]_E\rangle \simeq \F_p$. 

Clearly, $\langle[b]_E\rangle \subseteq (W^*)^G$.
From \eqref{eq:3} one gets the relation 
\[
[\sigma_a(A)]_E= [A]_E [b]_E.
\]
This implies that $[A]_E$  is not in $(W^*)^G$. Hence $\dim_{\F_p}\langle [b]_E,[A]_E\rangle=2$.

From \eqref{eq:4} one gets the relation 
\[  [\sigma_c(C)]_E= [C]_E [b]_E.
\]
This implies that $[C]_E$ is not in $(W^*)^{\sigma_c}$. But we have $\langle [b]_E,[A]_E\rangle\subseteq (W^*)^{\sigma_c}$. Hence
\[
\dim_{\F_p}\langle [b]_E,[A]_E,[C]_E\rangle =3.
\]

Observe that the element $(\sigma_a-1)(\sigma_c-1)$ annihilates the $\F_p[G]$-module $\langle [b]_E,[A]_E,[C]_E\rangle$, while by  \eqref{eq:1} and \eqref{eq:3} one has
\[
(\sigma_a-1)(\sigma_c-1)[\delta]_E= \frac{\sigma_a([A]_E)}{[A]_E}= [b]_E.
\]
Therefore one has
\[
\dim_{\F_p}W^*=\dim_{\F_p}\langle [b]_E,[A]_E,[C]_E,[\delta]_E\rangle =4.
\]
\\
\\
Let $H^{a,b}=F(\sqrt[p]{a},\sqrt[p]{A},\sqrt[p]{b})$ and $H^{b,c}=F(\sqrt[p]{c},\sqrt[p]{C},\sqrt[p]{b})$. 
Let 
\[N:=H^{a,b}H^{b,c}=F(\sqrt[p]{a},\sqrt[p]{c},\sqrt[p]{b},\sqrt[p]{A},\sqrt[p]{C})=E(\sqrt[p]{b},\sqrt[p]{A},\sqrt[p]{C}).\] 
Then $N/F$ is a Galois extension of degree $p^5$. 
This is because  ${\rm Gal}(N/E)$ is dual to the $\F_p[G]$-submodule $\langle  [b]_E,[A]_E,[C]_E \rangle$ via Kummer theory, 
and the proof of the claim above shows that $\dim_{\F_p}\langle [b]_E,[A]_E,[C]_E\rangle =3$. We have the following commutative diagram
\[
\xymatrix{
{\rm Gal}(N/F) \ar@{->>}[r] \ar@{->>}[d] & {\rm Gal}(H^{a,b}/F) \ar@{->>}[d]\\
{\rm Gal}(H^{b,c}/F) \ar@{->>}[r] & {\rm Gal}(F(\sqrt[p]{b})/F).
}
\]
So we have a homomorphism $\eta$ from ${\rm Gal}(N/F)$ to the pull-back ${\rm Gal}(H^{b,c}/F)\times_{{\rm Gal}(F(\sqrt[p]{b})/F)} {\rm Gal}(H^{a,b}/F)$:
\[
\eta\colon {\rm Gal}(N/F) \longrightarrow {\rm Gal}(H^{b,c}/F)\times_{{\rm Gal}(F(\sqrt[p]{b})/F)} {\rm Gal}(H^{a,b}/F),
\]
which make the obvious diagram commute.
We claim that $\eta$ is injective. Indeed, let $\sigma$ be an element in $\ker\eta$. Then $\sigma\mid_{H^{a,b}}=1$ in ${\rm Gal}(H^{a,b}/F)$, and $\sigma\mid_{H^{b,c}}=1$ in ${\rm Gal}(H^{b,c}/F)$. Since $N$ is the compositum of $H^{a,b}$ and $H^{b,c}$, this implies that $\sigma=1$, as desired. 

Since $|{\rm Gal}(H^{b,c}/F)\times_{{\rm Gal}(F(\sqrt[p]{b})/F)} {\rm Gal}(H^{a,b}/F)|=p^5=|{\rm Gal}(N/F)|$, we see that $\eta$ is actually an isomorphism.
As in the proof of Proposition~\ref{prop:Heisenberg extension}, we can choose an extension $\sigma_a\in {\rm Gal}(H^{a,b}/F)$ of $\sigma_a\in {\rm Gal}(F(\sqrt[p]{a},\sqrt[p]{b})/F)$ (more precisely, of $\sigma_a\hspace*{-5pt}\mid_{F(\sqrt[p]{a},\sqrt[p]{b})}\in {\rm Gal}(F(\sqrt[p]{a},\sqrt[p]{b})/F)$) in such a way that 
\[
\sigma_a(\sqrt[p]{A})= \sqrt[p]{A} \frac{\sqrt[p]{b}}{\alpha}.
\]
Since the square commutative diagram above is a pull-back, we can choose an extension  $\sigma_a\in {\rm Gal}(N/F)$ of $\sigma_a\in {\rm Gal}(H^{a,b}/F)$ in such a way that
\[
 \sigma_a\mid_{H^{b,c}}=1.
\]
Now we can choose any extension $\sigma_a\in {\rm Gal}(M/F)$ of $\sigma_a\in{\rm Gal}(N/F)$. Then we have
\[
\sigma_a(\sqrt[p]{A})= \sqrt[p]{A} \frac{\sqrt[p]{b}}{\alpha} \; \text{ and }  \sigma_a\mid_{H^{b,c}}=1.
\]

Similarly, we can choose an extension  $\sigma_c\in {\rm Gal}(M/F)$ of $\sigma_c\in {\rm Gal}(F(\sqrt[p]{b},\sqrt[p]{c})/F)$ in such a way that
\[
\sigma_c(\sqrt[p]{C})= \sqrt[p]{C} \frac{\sqrt[p]{b}}{\gamma}, \text{ and } \sigma_c\mid_{H^{a,b}}=1.
\]

We define $\sigma_b\in {\rm Gal}(M/E)$ to be the element which is dual to $[b]_E$ via Kummer theory. In other words, we require that
\[
\sigma_b(\sqrt[p]{b})=\xi \sqrt[p]{b},
\]
and $\sigma_b$ acts trivially on $\sqrt[p]{A}$, $\sqrt[p]{C}$ and $\sqrt[p]{\delta}$. We consider $\sigma_b$ as an element in ${\rm Gal}(M/F)$, then it is clear that $\sigma_b$ is an extension of $\sigma_b\in {\rm Gal}(F(\sqrt[p]{a},\sqrt[p]{b},\sqrt[p]{c})/F)$.
Let $W={\rm Gal}(M/E)$, and let $H={\rm Gal}(M/F)$, then we have the following exact sequence
\[
1\to W\to H\to G\to 1.
\]
By Kummer theory, it follows that $W$ is dual to $W^*$, and hence $W\simeq (\Z/p\Z)^4$. In particular, we have $|H|=p^6$.

Recall that from \cite[Theorem 1]{BD}, we know that the group $\U_4(\F_p)$ has a presentation with generators $s_1,s_2,s_3$ subject to the following relations
\[
\label{eq:R}
\tag{R}
\begin{aligned}
s_1^p=s_2^p=s_3^p=1,\\
[s_1,s_3]=1,\\
[s_1,[s_1,s_2]]=[s_2,[s_1,s_2]]=1,\\
[s_2,[s_2,s_3]]=[s_3,[s_2,s_3]]=1,\\
[[s_1,s_2],[s_2,s_3]]=1.
\end{aligned}
\]
Note that  $|\Gal(M/F)|=p^6$.
So in order to show that $\Gal(M/F)\simeq \U_4(\F_p)$, we shall show that $\sigma_a,\sigma_b$ and $\sigma_c$ generate $\Gal(M/F)$ and that they satisfy these above relations.
\\
\\
{\bf Claim:} The elements $\sigma_a,\sigma_b$ and $\sigma_c$ generate $\Gal(M/F)$.\\
{\it Proof of Claim:} Let $K$ be the maximal $p$-elementary subextension of $M$. Note that $\Gal(M/F)$ is a $p$-group. So in order to show that $\sigma_a,\sigma_b$ and $\sigma_c$ generate $\Gal(M/F)$, we only need to show that (the restrictions of) these elements generate $\Gal(K/F)$ by the Burnside basis theorem. (See e.g. \cite[Theorem 12.2.1]{Ha} or \cite[Theorem 4.10]{Ko}.)
 We shall now determine the field $K$.
By Kummer theory, $K=F(\sqrt[p]{\Delta})$, where $\Delta= (F^\times\cap {M^\times}^p)/(F^\times)^p$. 
Let $[f]_F$ be any element in $\Delta$, where $f\in F^\times \cap (M^\times)^p\subseteq E^\times \cap  (M^\times)^p$. 
By Kummer theory, one has
$
W^*=(E^\times \cap (M^\times)^p)/(E^\times)^p. 
$ 
Hence we can write
\[
[f]_E=[\delta]_E^{\epsilon_\delta} [A]_E^{\epsilon_A} [C]_E^{\epsilon_C} [b]_E^{\epsilon_b},
\]
where $\epsilon_\delta,\epsilon_A,\epsilon_C,\epsilon_b\in \Z$. By applying $(\sigma_a-1)(\sigma_c-1)$ on $[f]_E$ we get
$
[1]_E=[b]_E^{\epsilon_\delta}.
$
(See the proof of the first claim of this proof.)
Thus $\epsilon_\delta $ is divisible by $p$, and one has 
\[
[f]_E=[A]_E^{\epsilon_A} [C]_E^{\epsilon_C} [b]_E^{\epsilon_b}.
\]  
By applying $\sigma_a-1$ on both sides of this equation, we get
$
[1]_E=[b]_E^{\epsilon_A}.
$
Thus $\epsilon_A$ is divisible by $p$.  Similarly, $\epsilon_C$ is also divisible by $p$. Hence $f=b^{\epsilon_b}e^p$ for some $e\in E$. Since $b$ and $f$ are in $F$, $e^p$ is in $F^\times \cap (E^\times)^p$ and $[e^p]_F$ is in $\langle [a]_F,[c]_F\rangle$. Therefore $[f]_F$ is in $\langle [a]_F,[b]_F,[c]_F\rangle$ and 
\[
\Delta= \langle [a]_F,[b]_F,[c]_F\rangle.
\]
So $K=F(\sqrt[p]{a},\sqrt[p]{b},\sqrt[p]{c})$. Then it is clear that $\sigma_a,\sigma_b$ and $\sigma_c$ generate $\Gal(K/F)$ and the claim follows.
\\
\\
{\bf Claim:} The order of $\sigma_a$ is $p$.\\
{\it Proof of Claim:}
As in the proof of Proposition~\ref{prop:Heisenberg extension}, we see that $\sigma_a^p(\sqrt[p]{A})=\sqrt[p]{A}$.

Since $\sigma_a(\delta)=\delta C C_2^{-p}$ (equation~\eqref{eq:2}), one has $\sigma_a(\sqrt[p]{\delta})=\xi^i \sqrt[p]{\delta} \sqrt[p]{C} C_2^{-1}$ for some $i\in \Z$. This implies that
\[
\begin{aligned}
\sigma_a^2(\sqrt[p]{\delta})&= \xi^i\sigma_a(\sqrt[p]{\delta}) \sigma_a(\sqrt[p]{C}) \sigma_a(C_2)^{-1}\\
&= \xi^{2i} \sqrt[p]{\delta} (\sqrt[p]{C})^2 C_2^{-1}\sigma_a(C_2)^{-1}.
\end{aligned}
\]
Inductively, we obtain
\[
\begin{aligned}
\sigma_a^{p}(\sqrt[p]{\delta})&= \xi^{pi}\sqrt[p]{\delta}(\sqrt[p]{C})^p N_{\sigma_a}(C_2)^{-1}\\
&= \sqrt[p]{\delta}(C )N_{\sigma_a}(C_2)^{-1}\\
& =\sqrt[p]{\delta}. 
\end{aligned}
\]
Therefore, we can conclude that $\sigma_a^p=1$, and $\sigma_a$ is of order $p$. 
\\
\\
{\bf Claim:} The order of $\sigma_b$ is $p$.\\
{\it Proof of Claim:} This is clear because $\sigma_b$ acts trivially on $\sqrt[p]{A},\sqrt[p]{C}, \delta$, $E$ and $\sigma_b(\sqrt[p]{b})=\xi\sqrt[p]{b}$.
\\
\\
{\bf Claim:} The order of $\sigma_c$ is $p$.\\
{\it Proof of Claim:} As in the proof of Proposition~\ref{prop:Heisenberg extension}, we see that $\sigma_c^p(\sqrt[p]{C})=\sqrt[p]{C}$.
 
Since $\sigma_c(\delta)=\delta A C_1^{-p}$ (equation~\eqref{eq:1}), one has $\sigma_c(\sqrt[p]{\delta})=\xi^j \sqrt[p]{\delta} \sqrt[p]{A} C_1^{-1}$ for some $j\in \Z$. This implies that
\[
\begin{aligned}
\sigma_c^2(\sqrt[p]{\delta})&= \xi^j\sigma_c(\sqrt[p]{\delta}) \sigma_c(\sqrt[p]{A}) \sigma_c(C_1)^{-1}\\
&= \xi^{2j} \sqrt[p]{\delta} (\sqrt[p]{A})^2 C_1^{-1}\sigma_c(C_1)^{-1}.
\end{aligned}
\]
Inductively, we obtain
\[
\begin{aligned}
\sigma_c^{p}(\sqrt[p]{\delta})&= \xi^{pj}\sqrt[p]{\delta}(\sqrt[p]{A})^p N_{\sigma_c}(C_1)^{-1}\\
&= \sqrt[p]{\delta}(A )N_{\sigma_c}(C_1)^{-1}\\
& = \sqrt[p]{\delta}.
\end{aligned}
\]
Therefore, we can conclude that $\sigma_c^p=1$, and $\sigma_c$ is of order $p$.   
\\
\\
{\bf Claim:} $[\sigma_a,\sigma_c]=1$.\\
{\it Proof of Claim:}  
It is enough to check that $\sigma_a\sigma_c(\sqrt[p]{\delta})=\sigma_c\sigma_a(\sqrt[p]{\delta})$.

We have
\[
\begin{aligned}
\sigma_a\sigma_c(\sqrt[p]{\delta})&= \sigma_a(\xi^j \sqrt[p]{\delta} \sqrt[p]{A} C_1^{-1})\\
&= \xi^j \sigma_a(\sqrt[p]{\delta})\sigma_a(\sqrt[p]{A})\sigma_a(C_1)^{-1}\\
&= \xi^j \xi^i \sqrt[p]{\delta} \sqrt[p]{C} C_2^{-1} \sqrt[p]{A} \frac{\sqrt[p]{b}}{\alpha}\sigma_a(C_1)^{-1}\\
&=\xi^{i+j} \sqrt[p]{\delta} \sqrt[p]{C} \sqrt[p]{A} \frac{\sqrt[p]{b}}{\alpha} (\sigma_a(C_1)C_2)^{-1}\\
&= \xi^{i+j} \sqrt[p]{\delta} \sqrt[p]{C} \sqrt[p]{A} \frac{\sqrt[p]{b}}{\alpha} \frac{(C_1\sigma_c(C_2))^{-1}}{B}\\
&=\xi^{i+j} \sqrt[p]{\delta} \sqrt[p]{C} \sqrt[p]{A} \frac{\sqrt[p]{b}}{\gamma}(C_1\sigma_c(C_2))^{-1}.
\end{aligned}
\]
On the other hand, we have
\[
\begin{aligned}
\sigma_c\sigma_a(\sqrt[p]{\delta})&= \sigma_c(\xi^i \sqrt[p]{\delta} \sqrt[p]{C} C_2^{-1})\\
&= \xi^i \sigma_c(\sqrt[p]{\delta})\sigma_c(\sqrt[p]{C})\sigma_c(C_2)^{-1}\\
&= \xi^i \xi^j \sqrt[p]{\delta} \sqrt[p]{A} C_1^{-1} \sqrt[p]{C} \frac{\sqrt[p]{b}}{\gamma}\sigma_c(C_2)^{-1}\\
&=\xi^{i+j} \sqrt[p]{\delta} \sqrt[p]{A} \sqrt[p]{C} \frac{\sqrt[p]{b}}{\gamma} (C_1\sigma_c(C_2))^{-1}.
\end{aligned}
\]
Therefore, $\sigma_a\sigma_c(\sqrt[p]{\delta})=\sigma_c\sigma_a(\sqrt[p]{\delta})$, as desired. 
\\
\\
{\bf Claim:} $[\sigma_a,[\sigma_a,\sigma_b]]=[\sigma_b,[\sigma_a,\sigma_b]]=1$.\\
{\it Proof of Claim:} Since $G$ is abelian, it follows that $[\sigma_a,\sigma_b]$ is in $W$. Now both $\sigma_b$ and $[\sigma_a,\sigma_b]$ are in $W$. Hence $[\sigma_b,[\sigma_a,\sigma_b]]=1$ because $W$ is abelian.

Now we show that $[\sigma_a,[\sigma_a,\sigma_b]]=1$.
 Since the Heisenberg group $\U_3(\F_p)$ is a nilpotent group of nilpotent length 2, we see that $[\sigma_a,[\sigma_a,\sigma_b]]=1$ on $H^{a,b}$ and $H^{b,c}$. So it is enough to check that $[\sigma_a,[\sigma_a,\sigma_b]](\sqrt[p]{\delta})=\sqrt[p]{\delta}$.

From the choice of $\sigma_b$, we see that
\[
\sigma_b\sigma_a(\sqrt[p]{\delta})=\sigma_a(\sqrt[p]{\delta})= \sigma_a\sigma_b(\sqrt[p]{\delta}).
\]
Hence, $[\sigma_a,\sigma_b](\sqrt[p]{\delta})=\sqrt[p]{\delta}$. Since $\sigma_a$ and $\sigma_b$ act trivially on $\sqrt[p]{C}$, and $\sigma_b$ acts trivially on $E$, we see that
\[
[\sigma_a,\sigma_b](\sqrt[p]{C})=\sqrt[p]{C},\;\; \text{ and } [\sigma_a,\sigma_b](C_2^{-1})=C_2^{-1}.
\]
We have 
\[
\begin{aligned}
{[\sigma_a,\sigma_b]}\sigma_a (\sqrt[p]{\delta})&=[\sigma_a,\sigma_b](\xi^i\sqrt[p]{\delta}\sqrt[p]{C}C_2^{-1})\\
&= [\sigma_a,\sigma_b](\xi^i) [\sigma_a,\sigma_b](\sqrt[p]{\delta}) [\sigma_a,\sigma_b](\sqrt[p]{C}) [\sigma_a,\sigma_b](C_2^{-1})\\
&= \xi^i \sqrt[p]{\delta}\sqrt[p]{C}C_2^{-1}\\
&=\sigma_a(\sqrt[p]{\delta})\\
&=\sigma_a[\sigma_a,\sigma_b](\sqrt[p]{\delta}).
\end{aligned}
\]
Thus $[\sigma_a,[\sigma_a,\sigma_b]](\sqrt[p]{\delta})=\sqrt[p]{\delta}$, as desired. 
\\
\\
{\bf Claim:} $[\sigma_b,[\sigma_b,\sigma_c]]=[\sigma_c,[\sigma_b,\sigma_c]]=1$.\\
{\it Proof of Claim:} Since $G$ is abelian, it follows that $[\sigma_b,\sigma_c]$ is in $W$. Now both $\sigma_b$ and $[\sigma_b,\sigma_c]$ are in $W$. Hence $[\sigma_b,[\sigma_b,\sigma_c]]=1$ because $W$ is abelian.

Now we show that $[\sigma_c,[\sigma_b,\sigma_c]]=1$. Since the Heisenberg group $\U_3(\F_p)$ is a nilpotent group of nilpotent length 2, we see that $[\sigma_c,[\sigma_b,\sigma_c]]=1$ on $H^{a,b}$ and $H^{b,c}$. So it is enough to check that $[\sigma_c,[\sigma_b,\sigma_c]](\sqrt[p]{\delta})=\sqrt[p]{\delta}$.

From the choice of $\sigma_b$, we see that
\[
\sigma_b\sigma_c(\sqrt[p]{\delta})=\sigma_c(\sqrt[p]{\delta})= \sigma_c\sigma_b(\sqrt[p]{\delta}).
\]
Hence,  $[\sigma_b,\sigma_c](\sqrt[p]{\delta})=\sqrt[p]{\delta}$. 
Since $\sigma_b$ and $\sigma_c$ act trivially on $\sqrt[p]{A}$, and $\sigma_b$ acts trivially on $E$, we see that
\[
[\sigma_b,\sigma_c](\sqrt[p]{A})=\sqrt[p]{A},\;\; \text{ and } [\sigma_b,\sigma_c](C_1^{-1})=C_1^{-1}.
\]
We have 
\[
\begin{aligned}
{[\sigma_b,\sigma_c]}\sigma_c (\sqrt[p]{\delta})&=[\sigma_b,\sigma_c](\xi^j\sqrt[p]{\delta}\sqrt[p]{A}C_1^{-1})\\
&= [\sigma_b,\sigma_c](\xi^j) [\sigma_b,\sigma_c](\sqrt[p]{\delta}) [\sigma_b,\sigma_c](\sqrt[p]{A}) [\sigma_b,\sigma_c](C_1^{-1})\\
&= \xi^j \sqrt[p]{\delta}\sqrt[p]{A}C_1^{-1}\\
&=\sigma_c(\sqrt[p]{\delta})\\
&=\sigma_c[\sigma_a,\sigma_b](\sqrt[p]{\delta}).
\end{aligned}
\]
Thus $[\sigma_c,[\sigma_b,\sigma_c]](\sqrt[p]{\delta})=\sqrt[p]{\delta}$, as desired. 
\\
\\
{\bf Claim:} $[[\sigma_a,\sigma_b],[\sigma_b,\sigma_c]]=1$.\\
{\it Proof of Claim}: Since $G$ is abelian,  $[\sigma_a,\sigma_b]$ and $[\sigma_b,\sigma_c]$ are in $W$. Hence $[[\sigma_a,\sigma_b],[\sigma_b,\sigma_c]]=1$ because $W$ is abelian.

  An explicit isomorphism $\varphi\colon {\rm Gal}(M/F)\to \U_4(\F_p)$ may be defined as
\[
\sigma_a \mapsto \begin{bmatrix}
1& 1 & 0 & 0\\
0& 1 & 0 & 0\\
0& 0 & 1 & 0\\
0& 0 & 0 & 1
\end{bmatrix}, \; \;
\sigma_b\mapsto  \begin{bmatrix}
1& 0 & 0 & 0\\
0& 1 & 1 & 0\\
0& 0 & 1 & 0\\
0& 0 & 0 & 1
\end{bmatrix}, \;\;
 \sigma_c\mapsto \begin{bmatrix}
1& 0 & 0 & 0\\
0& 1 & 0 & 0\\
0& 0 & 1 & 1\\
0& 0 & 0 & 1
\end{bmatrix}.
\]
\end{proof}
\begin{ex}
\label{ex:p=2}
  Let the notation and assumption be as in Lemma~\ref{lem:modification}.  Let us consider the case $p=2$. In Lemma~\ref{lem:C1C2}, we can choose $e=\dfrac{\alpha}{\alpha+\gamma}$. (Observe that $\alpha+\gamma\not=0$.) In fact, one can easily check that
\[
\sigma_a\sigma_c(\frac{\alpha}{\alpha+\gamma})=\frac{\gamma}{\alpha} \frac{\alpha}{\alpha+\gamma}.
\] 
\\
\\
(1) If we choose $C_1=\sigma_c(e)$ and  $C_2=e^{-1}$ as in Lemma~\ref{lem:C1C2} part (1), then  we have
\[
\begin{aligned}
A=N_{\sigma_c}(C_1)&=N_{\sigma_c}(e)= \frac{\alpha^2\gamma}{(\alpha+\gamma)(\alpha\gamma+b)},\\
C= N_{\sigma_a}(C_2)&=N_{\sigma_a}(e^{-1})=\frac{(\alpha+\gamma)(\alpha\gamma+b)}{b\alpha}.
\end{aligned}
\]

In Lemma~\ref{lem:modification}, we can choose $\delta=e^{-1}=\dfrac{\alpha+\gamma}{\alpha}$. In fact, we have
\[
\begin{aligned}
\frac{\sigma_c(\delta)}{\delta}&=\sigma_c(e)^{-1}e= \sigma_c(e)^{-2} e\sigma_c(e)= C_1^{-2} N_{\sigma_c}(e)= AC_1^{-2},\\
\frac{\sigma_a(\delta)}{\delta}&=\sigma_a(e^{-1})e= e^{-1}\sigma_a(e^{-1}) e^2= N_{\sigma_a}(e^{-1})C_2^{-2}= CC_2^{-2}.
\end{aligned}
\]
Therefore
\[
\begin{aligned}
M&= F(\sqrt{b},\sqrt{A},\sqrt{C},\sqrt{\delta})= F(\sqrt{b},\sqrt{ \frac{\alpha^2\gamma}{(\alpha+\gamma)(\alpha\gamma+b)}},\sqrt{\frac{(\alpha+\gamma)(\alpha\gamma+b)}{b\alpha}},\sqrt{\dfrac{\alpha+\gamma}{\alpha}})\\
&=F(\sqrt{b},\sqrt{\frac{\alpha+\gamma}{\alpha}},\sqrt{\alpha\gamma+b},\sqrt{\alpha\gamma}).
\end{aligned}
\]
\\
\\
(2) If we choose $C_1=e=\dfrac{\alpha}{\alpha+\gamma}$ and  $C_2=eB=\dfrac{\gamma}{\alpha+\gamma}$ as in Lemma~\ref{lem:C1C2} part (2), then  we have
\[
\begin{aligned}
A=N_{\sigma_c}(C_1)&=N_{\sigma_c}(e)= \frac{\alpha^2\gamma}{(\alpha+\gamma)(\alpha\gamma+b)},\\
C= N_{\sigma_a}(C_2)&=N_{\sigma_a}(eB)=\frac{\gamma^2\alpha}{(\alpha+\gamma)(\alpha\gamma+b)}.
\end{aligned}
\]

In Lemma~\ref{lem:modification}, we can choose $\delta=(\alpha+\gamma)^{-1}$. In fact, we have
\[
\begin{aligned}
\frac{\sigma_c(\delta)}{\delta}&=\frac{\gamma(\alpha+\gamma)}{\alpha\gamma+b}= AC_1^{-2},\\
\frac{\sigma_a(\delta)}{\delta}&=\frac{\alpha(\alpha+\gamma)}{\alpha\gamma+b}= CC_2^{-2}.
\end{aligned}
\]
Therefore
\[
M= F(\sqrt{b},\sqrt{A},\sqrt{C},\sqrt{\delta})= F(\sqrt{b},\sqrt{\frac{\alpha^2\gamma}{\alpha\gamma+b}},\sqrt{\frac{\alpha\gamma^2}{\alpha\gamma+b}},\sqrt{\alpha+\gamma}).
\]
Observe also that $M$ is  the Galois closure of $E(\sqrt{\delta})=F(\sqrt{a},\sqrt{c},\sqrt{\alpha+\gamma})$.
\end{ex}
%%%%%%%%%%%%%%%%%%%%%%%%%%%%%%%%%%%%%%%%%%%%%%%%%%%%%
%%%%%%%%%%%%%%%%%%%%%%%%%%%%%%%%%%%%%%%%%%%%%%%%%%%%
\subsection{Fields of characteristic not $p$}
Let $F_0$ be an arbitrary field of characteristic $\not=p$. We fix a primitive $p$-th root of unity $\xi$, and let $F=F_0(\xi)$. Then $F/F_0$ is a cyclic extension of degree $d=[F:F_0]$. Observe that $d$ divides $p-1$. We choose an integer $\ell$ such that $d\ell\equiv 1\bmod p$.
Let $\sigma_0$ be a generator of $H:={\rm Gal}(F/F_0)$. Then $\sigma_0(\xi)=\xi^e$ for an $e\in \Z\setminus p\Z$.

Let $\chi_1,\chi_2,\chi_3$ be elements in ${\rm Hom}(G_{F_0},\F_p)=H^1(G_{F_0},\F_p)$. We assume that   $\chi_1,\chi_2,\chi_3$ are $\F_p$-linearly independent and 
$\chi_1 \cup \chi_2 = \chi_2 \cup \chi_3 =0$.  By \cite[Lemma 2.6]{MT4}, the homomorphism  $(\chi_1,\chi_2,\chi_3)\colon G_{F_0}\to (\F_p)^3$ is surjective. 
  Let $L_0$ be the fixed field of $(F_0)^s$ under the kernel of the surjection $(\chi_1,\chi_2,\chi_3)\colon G_{F_0}\to (\F_p)^3$. Then $L_0/F_0$ is Galois with ${\rm Gal}(L_0/F_0)\simeq (\F_p)^3$. We shall construct a Galois extension $M_0/F_0$ such that ${\rm Gal}(M_0/F_0)\simeq \U_4(\F_p)$ and $M_0$ contains $L_0$.

The restrictions $\res_{G_F}(\chi_1),\res_{G_F}(\chi_2),\res_{G_F}(\chi_3)$ are elements in  ${\rm Hom}(G_F,\F_p)$. They are $\F_p$-linearly independent and 
$\res_{G_F}(\chi_1) \cup \res_{G_F}(\chi_2) =\res_{G_F}(\chi_2) \cup\res_{G_F}(\chi_3) =0$. 
By Kummer theory there exist $a,b,c$ in $F^\times$ such that $\res_{G_F}(\chi_1)=\chi_a$, $\res_{G_F}(\chi_2)=\chi_b$, $\res_{G_F}(\chi_3)=\chi_c$. Then we have
$(a,b)=(b,c)=0$ in $H^2(G_F,\F_p)$.

Let $L=L_0(\xi)$. Then $L=F(\sqrt[p]{a},\sqrt[p]{b},\sqrt[p]{c})$, and  $L/F$ is Galois with ${\rm Gal}(L/F)\simeq {\rm Gal}(L_0/F_0)\simeq (\F_p)^3$.
\\
\\
{\bf Claim 1:} $L/F_0$ is Galois with ${\rm Gal}(L/F_0)\simeq {\rm Gal}(F/F_0)\times {\rm Gal}(L/F)$.\\
{\it Proof of Claim}: Since $L_0/F_0$ and $F/F_0$ are Galois extensions of relatively prime degrees, the claim follows.
\\
\\
We define $\displaystyle \Phi:= \ell[\sum_{i=0}^{d-1}e^i\sigma_0^{-i}]\in \Z[H]$.  The group ring $\Z[H]$ acts on $F$ in the obvious way, and if we let $H$ act trivially on $L_0$ we get an action on $L$ also. Then $\Phi$ determines a map
\[
\Phi\colon L\to L, x\mapsto \Phi(x).
\]
For convenience, we shall denote $\tilde{x}:=\Phi(x)$.

The claim above implies that $\Phi\sigma=\sigma\Phi$ for every $\sigma\in {\rm Gal}(L/F)$. 
\\
\\
{\bf Claim 2:} We have $\tilde{a}=a$ modulo $(F^\times)^p$; $\tilde{b}=b$ modulo $(F^\times)^p$, $\tilde{c}=c$ modulo $(F^\times)^p$.\\
{\it Proof of Claim:} A similar argument as in the proof of Claim 1 shows that $F(\sqrt[p]{a})/F_0$ is Galois with ${\rm Gal}(F(\sqrt[p]{a})/F_0)={\rm Gal}(F(\sqrt[p]{a})/F)\times {\rm Gal}(F/F_0)$. Since both groups ${\rm Gal}(F(\sqrt[p]{a})/F)$ and ${\rm Gal}(F/F_0)$ are cyclic and of coprime orders, we see that the extension $F(\sqrt[p]{a})/F_0$ is cyclic. By Albert's result (see \cite[pages 209-211]{Alb} and \cite[Section 5]{Wat}), we have $\sigma_0 a =a^e $ modulo $(F^\times)^p$. Hence for all integers $i$, $\sigma_0^i(a)=a^{e^i}\bmod (F^\times)^p$. Thus $\sigma_0^{-i}(a^{e^i})=a\bmod (F^\times)^p$.
Therefore, we have
\[
\tilde{a}=\Phi(a)=\left[ \prod_{i=0}^{d-1}\sigma_0^{-i}(a^{e^i})\right]^\ell= \left[\prod_{i=0}^{d-1} a \right]^\ell= a^{d\ell}=a \bmod (F^\times)^p.
\]
Similarly, we have $\tilde{b}=b$ modulo $(F^\times)^p$, $\tilde{c}=c$ modulo $(F^\times)^p$.
\\
\\
{\bf Claim 3:} For every $x\in L$, we have $\dfrac{\sigma_0\tilde{x}}{\tilde{x}^e}= \sigma_0(x^{\ell(1-e^d)/p})^p \in L^p$.\\
{\it Proof of Claim:} This follows from the following identity in the group ring $\Z[H]$,
\[
(\sigma_0-e)(\sum_{i=0}^{d-1}e^i\sigma_0^{-i})= \sigma_0(1-e^d)\equiv 0\bmod p.
\]
\\
\\
By our construction of Galois $\U_4(\F_p)$-extensions  over fields containing a primitive $p$-th root of unity (see Subsection~\ref{subsec:with p-primitive}), we have $\alpha,\gamma,B,...,A,C,\delta$ such that if we let 
$M:= L(\sqrt[p]{A},\sqrt[p]{C},\sqrt[p]{\delta})$, then $M/F$ is a Galois $\U_4(\F_p)$-extension. 
We set $\tilde{M}:=L(\sqrt[p]{\tilde{A}},\sqrt[p]{\tilde{C}},\sqrt[p]{\tilde{\delta}})$.
\\
\\
{\bf Claim 4:} $\tilde{M}/F$ is Galois with ${\rm Gal}(\tilde{M}/F)\simeq \U_4(\F_p)$.\\
{\it Proof of Claim}: Since $\Phi$ commutes with every $\sigma\in {\rm Gal}(L/F)$, this implies that $\tilde{M}/F$ is Galois. 
This, together with Claim 2, also implies that ${\rm Gal}(\tilde{M}/F)\simeq \U_4(\F_p)$  because the construction of $\tilde{M}$ over $F$ is obtained in the same way as in the construction of $M$, except that we replace the data $\{a,b,c, \alpha,\gamma, B,...\}$ by their "tilde" counterparts $\{\tilde{a},\tilde{b},\tilde{c}, \tilde{\alpha},\tilde{\gamma},\tilde{B},...\}$.
\\
\\
{\bf Claim 5:} $\tilde{M}/F_0$ is Galois with ${\rm Gal}(\tilde{M}/F_0)\simeq {\rm Gal}(\tilde{M}/F)\times {\rm Gal}(F/F_0)$.\\
{\it Proof of Claim}: By Claim 3, we see that $\sigma_0 \tilde{x}=\tilde{x}^e $ modulo $(L^\times)^p$ for every $\tilde{x}$ in the $\F_p$-vector subspace $\tilde{W^*}$ of $L^\times/(L^\times)^p$ generated by $\tilde{A}$, $\tilde{C}$, and $\tilde{\delta}$. Hence $\tilde{W^*}$ is an $\F_p[{\rm Gal}(L/F_0)]$-module.
Therefore $\tilde{M}/F_0$ is Galois by Kummer theory. 

We also have the following exact sequence of groups
\[
1\to {\rm Gal}(\tilde{M}/F)\to {\rm Gal}(\tilde{M}/F_0)\to {\rm Gal}(F/F_0)\to 1.
\]
Since $|{\rm Gal}(\tilde{M}/F)|$ and $|{\rm Gal}(F/F_0)|$ are coprime, the above sequence is split by Schur-Zassenhaus's theorem. (See \cite[IV.7,Theorem 25]{Za}.)
The Galois group ${\rm Gal}(\tilde{M}/F_0)$ is the semidirect product of ${\rm Gal}(\tilde{M}/F)$ and $H={\rm Gal}(F/F_0)$, with $H$ acting on ${\rm Gal}(\tilde{M}/F)$ by conjugation. We need to show that this product is in fact direct, i.e., that the action of $H$ on ${\rm Gal}(\tilde{M}/F)$ is trivial.
Note that $H$ has an order coprime to $p$, and $H$ acts trivially on ${\rm Gal}(L/F)$ (see Claim 1) which is the quotient of ${\rm Gal}(\tilde{M}/F)$ by its Frattini subgroup. 
Then a result of P. Hall (see \cite[Theorem 12.2.2]{Ha})  implies that $H$ acts trivially on ${\rm Gal}(\tilde{M}/F)$.

From the discussion above we obtain  the following result.
\begin{thm}
\label{thm:construction char not p} 
Let the notation be as above. 
Let $M_0$ be the fixed field of $\tilde{M}$ under the subgroup of ${\rm Gal}(\tilde{M}/F_0)$ which is isomorphic to ${\rm Gal}(F/F_0)$. 
Then $M_0/F_0$ is Galois with 
${\rm Gal}(M_0/F_0)\simeq {\rm Gal}(\tilde{M}/F)\simeq \U_4(\F_p)$, and $M_0$ contains $L_0$.
\end{thm}
\begin{proof} Claim 5 above implies that $M_0/F_0$ is Galois with 
${\rm Gal}(M_0/F_0)\simeq {\rm Gal}(\tilde{M}/F)\simeq \U_4(\F_p)$. Since $H\simeq {\rm Gal}(\tilde{M}/M_0)$ acts trivially on $L_0$, we see that $M_0$ contains $L_0$.

Let $\sigma_1:=\sigma_a|_{M_0}$, $\sigma_2:=\sigma_b|_{M_0}$ and $\sigma_3:=\sigma_c|_{M_0}$. 
  Then $\sigma_1,\sigma_2$ and $\sigma_3$ generate ${\rm Gal}(M_0/F_0)\simeq \U_4(F_p)$. We also have
  \[
  \begin{aligned}
  \chi_1(\sigma_1)=1,   \chi_1(\sigma_2)=0,   \chi_1(\sigma_3)=0;\\
  \chi_2(\sigma_1)=0,   \chi_2(\sigma_2)=1,   \chi_2(\sigma_3)=0;\\
  \chi_3(\sigma_1)=0,   \chi_3(\sigma_2)=0,   \chi_3(\sigma_3)=1.
  \end{aligned}
  \]
  (Note that for each $i=1,2,3$, $\chi_i$ is trivial on ${\rm Gal}(M/M_0)$, hence $\chi_i(\sigma_j)$ makes sense for every $j=1,2,3$.)
  An explicit isomorphism $\varphi\colon {\rm Gal}(M_0/F_0)\to \U_4(\F_p)$ may be defined as
\[
\sigma_1 \mapsto \begin{bmatrix}
1& 1 & 0 & 0\\
0& 1 & 0 & 0\\
0& 0 & 1 & 0\\
0& 0 & 0 & 1
\end{bmatrix}, \; \;
\sigma_2\mapsto  \begin{bmatrix}
1& 0 & 0 & 0\\
0& 1 & 1 & 0\\
0& 0 & 1 & 0\\
0& 0 & 0 & 1
\end{bmatrix}, \;\;
 \sigma_3\mapsto \begin{bmatrix}
1& 0 & 0 & 0\\
0& 1 & 0 & 0\\
0& 0 & 1 & 1\\
0& 0 & 0 & 1
\end{bmatrix}.
\]

\end{proof}
%%%%%%%%%%%%%%%%%%%%%%%%%%%%%%%%%%%%%%%%%%%%%%%%%%%%
%%%%%%%%%%%%%%%%%%%%%%%%%%%%%%%%%%%%%%%%%%%%%%%%%%%%%%%%%
%%%%%%%%%%%%%%%%%%%%%%%%%%%%%%%%%%%%%%%%%%%%%%%%%%%%%%%%%%%%%
\section{The construction of $\U_4(\F_p)$-extensions: The case of characteristic $p$}
In this section we assume that $F$ is of characteristic $p>0$. Although by a theorem of Witt (see \cite{Wi} and \cite[Chapter 9, Section 9.1]{Ko}), we know that the Galois group of the maximal $p$-extension of $F$ is a free pro-$p$- group, finding specific constructions of Galois $p$-extensions over $F$ can still be challenging. The following construction of an explicit Galois extension $M/F$ with Galois group $\U_4(\F_p)$ is an analogue of the construction in Subsection 3.1 when we assumed that a $p$-th root of unity is in $F$. However we find the details interesting, and therefore for the convenience of the reader, we are including them here. Observe that even the case of the explicit construction of Heisenberg extensions of degree $p^3$ in characteristic $p$ is of interest. 
In the case  when $F$ has characteristic not $p$, the constructions of Heisenberg extensions of degree $p^3$ are now classical, important tools in Galois theory. We did not find any such constructions in the literature in the case of characteristic $p$. Nevertheless the construction in Subsection 4.2 seems to be simple, useful and aesthetically pleasing. What is even more surprising is that the field construction  of Galois $\U_4(\F_p)$-extensions over a field $F$ of characteristic $p$ in Subsection 4.3 is almost equally simple. We have to check more details to confirm the validity of this construction, but the construction of the required Galois extension $M$ itself, is remarkably simple.  
The possibility of choosing  generators in such a straightforward  manner (as described in Theorem~\ref{thm:construction char p}) is striking. It is interesting that the main construction in Section 3 carries over with necessary modifications in the case of characteristic $p$.
%%%%%%%%%%%%%%%%%%%%%%
\subsection{A brief review of Artin-Schreier theory} (For more details and the origin of this beautiful theory, see \cite{ASch}.)
Let $F$ be a field of characteristic $p>0$. Let $\wp(X)=X^p-X$ be the Artin-Schreier polynomial. For each $a$ in  $F$ of characteristic $p$, we let $\theta_a$ be a root of $\wp(X)=a$. 
We also denote $[a]_F$ to be the image of $a$ in $F/\wp(F)$. 
For each subgroup $U$ of $F/\wp(F)$,  let $F_U:=F(\theta_u: [u]_F\in U)$. Then the map $U\mapsto F_U$ is a bijection between subgroups of $F/\wp(F)$ and abelian extensions of $F$ of  exponent dividing $p$. There is a pairing 
\[
{\rm Gal}(F_U/F) \times U\to \F_p,
\]
defined by $\langle \sigma, a\rangle = \sigma(\theta_a)-\theta_a$,  
which is independent of the choice of root $\theta_a$. Artin-Schreier theory says that this pairing is non-degenerate. 

Now assume that $F/k$ is a finite Galois extension. The Galois group ${\rm Gal}(F/k)$  acts naturally on $F/\wp(F)$. As an easy exercise, one can show that such an extension $F_U$, where $U$ is a subgroup of $F/\wp(F)$, is Galois over  $k$ if and only if $U$ is actually an $\F_p[{\rm Gal}(F/k)]$-module.

\subsection{Heisenberg extensions in characteristic $p>0$}  
For each $a\in F$, let $\chi_a\in {\rm Hom}(G_F,\F_p)$ be the  corresponding element associated with $a$ via Artin-Schreier theory. Explicitly, $\chi_a$ is defined by 
\[
\chi_a(\sigma)= \sigma(\theta_a)-\theta_a.
\]

Assume that $a,b$ are elements in $F$, which are linearly independent modulo $\wp(F)$. Let $K= F(\theta_a,\theta_b)$. Then $K/F$ is a Galois extension whose Galois group is generated by $\sigma_a$ and $\sigma_b$. Here $\sigma_a(\theta_b)=\theta_b$, $\sigma_a(\theta_a)= \theta_{a}+1$; $\sigma_b(\theta_{a})=\theta_{a}$, $\sigma_b(\theta_{b})= \theta_{b}+1$. 

We set $A=b\theta_a$. Then 
\[
\sigma_a(A)= A+b,\, \text{ and } \sigma_b(A)=A.
\]
\begin{prop}
\label{prop:Heisenberg char p} Let the notation be as above. 
Let $L=K(\theta_A)$. Then $L/F$ is Galois whose Galois group is isomorphic to $\U_3(\F_p)$. 
\end{prop}
\begin{proof}
From $\sigma_a(A)- A= b \in \wp(K)$, and $\sigma_b(A)=A$, we see that $\sigma(A)-A \in \wp(K)$ for every $\sigma \in {\rm Gal}(K/F)$. This implies that the extension $L:=K(\theta_{A})/F$ is Galois. Let $\tilde{\sigma}_a\in {\rm Gal}(L/F)$ (resp. $\tilde{\sigma}_b\in {\rm Gal}(L/F)$) be an extension of $\sigma_a$ (resp. $\sigma_b$). Since $\sigma_b(A)=A$, we have $\tilde{\sigma}_b(\theta_A)=\theta_{A}+j$, for some $j\in \F_p$. Hence $\tilde{\sigma}_b^p(\theta_A)=\theta_{A}$. This implies that $\tilde{\sigma}_b$ is of order $p$.

On the other hand, we have
\[
\wp(\tilde{\sigma}_a(\theta_{A}))=\sigma_a(A) =A +b.
\]
Hence $\tilde{\sigma}_a(\theta_{A})= \theta_{A}+\theta_{b}+i$, for some $i\in \F_p$. Then
\[
\tilde{\sigma}_a^p(\theta_{A})= \theta_{A} +p\theta_b+pi=\theta_{A}.
\]
This implies that $\tilde{\sigma}_a$ is also of order $p$. We have
\[
\begin{aligned}
\tilde{\sigma}_a\tilde{\sigma}_b(\theta_{A}) &= \tilde{\sigma}_a(j+\theta_{A})=i+j+\theta_{A}+\theta_{b},\\
\tilde{\sigma}_b\tilde{\sigma}_a(\theta_{A}) &=\tilde{\sigma}_b(i+ \theta_{A}+\theta_{b})={i+j}+\theta_{A}+1+\theta_{b}.
\end{aligned}
\]
We set $ \tilde{\sigma}_{A}:= \tilde{\sigma}_a \tilde{\sigma}_b\tilde{\sigma}_a^{-1}\tilde{\sigma}_b^{-1}$. Then
\[
\tilde{\sigma}_{A}(\theta_{A})=\theta_{A}-1.
\]
This implies that $\tilde{\sigma}_{A}$ is of order $p$ and that ${\rm Gal}(L/F)$ is generated by $\tilde{\sigma}_a$ and $\tilde{\sigma}_b$. We also have
\[
\begin{aligned}
\tilde{\sigma}_a \tilde{\sigma}_{A}= \tilde{\sigma}_{A}\tilde{\sigma}_a, \;\text{ and } \tilde{\sigma}_b \tilde{\sigma}_{A}= \tilde{\sigma}_{A}\tilde{\sigma}_b.
\end{aligned}
\]
We can define an isomorphism $\varphi \colon {\rm Gal}(L/F)\to \U_3(\Z/p\Z)$ by letting
\[
\tilde{\sigma}_a \mapsto \begin{bmatrix}
1& 1 & 0 \\
0& 1 & 0 \\
0& 0 & 1
\end{bmatrix},
\tilde{\sigma}_b\mapsto 
\begin{bmatrix}
1& 0 & 0 \\
0& 1 & 1 \\
0& 0 & 1 
\end{bmatrix},
\tilde{\sigma}_{A}\mapsto
\begin{bmatrix}
1& 0 & 1 \\
0& 1 & 0 \\
0& 0 & 1 
\end{bmatrix}.
\]

Note that $[L:F]=p^3$. Hence there are exactly $p$ extensions of $\sigma_a\in {\rm Gal}(K/F)$ to the automorphisms in ${\rm Gal}(L/F)$ since $[L:K]=p^3/p^2=p$. Therefore for later use,  we can choose an extension of $\sigma_a\in {\rm Gal}(K/F)$, which we shall denote $\sigma_a\in {\rm Gal}(L/F)$ with a slight abuse of notation,  in such a way that $\sigma_a(\theta_{A})= \theta_{A}+ \theta_{b}$. 
\end{proof}
\begin{rmk} It is interesting to compare our choices of generators $A$ of Heisenberg extensions over given bicyclic extensions in the case of characteristic $p$ and the case when the base field in fact contains a primitive $p$-th root of unity.
 See the proofs of Proposition~\ref{prop:Heisenberg extension} and the proposition above. Although the form of $A$ in Proposition~\ref{prop:Heisenberg extension} is more complicated than the strikingly simple choice above, the basic principle for the search of a suitable $A$ is the same in both cases. We need to guarantee that $\sigma_a(A)/A\in (K^\times)^p$ and $\sigma_a(A)-A\in \wp(K)$  in order that $K(\sqrt[p]{A})$ and $K(\theta_A)$ are Galois over $F$.
 Further, in order to guarantee that $\sigma_a$ and $\sigma_b$ will not commute, we want  $\sigma_a(A)/A=bk^p$, for some $k\in K^\times$, where $\sigma_b$ acts trivially on $k$.
Similarly in the characteristic $p$ case we want $\sigma_a(A)-A=b+\wp(k)$, where $\sigma_b(k)=k$. If ${\rm char}(F)=p$, then this is always possible in a  simple way above with $k$ even being $0$. (See also \cite[Appendix A.1, Example]{JLY}, where the authors discuss a construction of quaternion $Q_8$-extensions over fields of characteristic $2$.)
In the case when $F$ contains a primitive $p$-th root of unity, the search for a suitable $A$ leads to the norm condition $N_{F(\sqrt[p]{a})/F}(\alpha)=b$. For some related considerations see \cite[Section 2]{MS1}.
\end{rmk}
%%%%%%%%%%%%%%%%%%%%%%%%%%%%%%
%%%%%%%%%%%%%%%%%%%%%%%%%%%%%
\subsection{Construction of Galois $\U_4(\F_p)$-extensions}
We assume that we are given  elements $a$, $b$ and $c$ in $F$ such that 
 $a$, $b$ and $c$ are linearly independent modulo $\wp(F)$. We shall construct a Galois $\U_4(\F_p)$-extension $M/F$ such that $M$ contains $F(\theta_{a},\theta_{b},\theta_{c})$.

First we note that  $F(\theta_{a},\theta_{b},\theta_{c})/F$ is a Galois extension with ${\rm Gal}(F(\theta_{a},\theta_{b},\theta_{c})/F)$ generated by $\sigma_a,\sigma_b,\sigma_c$. Here
\[
\begin{aligned}
\sigma_a(\theta_{a})&=1+ \theta_{a}, \sigma_a(\theta_{b})=\theta_{b}, \sigma_a(\theta_{c})=\theta_{c};\\
\sigma_b(\theta_{a})&=\theta_{a}, \sigma_b(\theta_{b})=1+ \theta_{b}, \sigma_b(\theta_{c})= \theta_{c};\\
\sigma_c(\theta_{a})&=\theta_{a}, \sigma_c(\theta_{b})=\theta_{b}, \sigma_c(\theta_{c})=1+ \theta_{c}.
\end{aligned}
\]
Recall that $A=b\theta_a$. We set $C:= b\theta_c$. 
We set $\delta:=(AC)/b=b\theta_a\theta_c \in E:=F(\theta_a,\theta_c)$. Then we have
\[
\begin{aligned}
\sigma_a(\delta)-\delta =b\sigma_a(\theta_a)\sigma_a(\theta_c)-b\theta_a\theta_c= b[\sigma_a(\theta_a)-\theta_a]\theta_c=b\theta_c=C,\\
\sigma_c(\delta)-\delta =b\sigma_c(\theta_a)\sigma_c(\theta_c)-b\theta_a\theta_c= b\theta_a[\sigma_c(\theta_c)-\theta_c]=b\theta_a=A.
\end{aligned}
\]
Finally set $G :={\rm Gal}(E/F).$

\begin{thm} 
\label{thm:construction char p}
 Let $M:= E(\theta_{\delta},\theta_{A},\theta_{C},\theta_{b})$. Then $M/F$ is a Galois extension, $M$ contains $F(\theta_{a},\theta_{b},\theta_{c})$, and ${\rm Gal}(M/F)\simeq \U_4(\F_p)$.
\end{thm}
\begin{proof}
Let $W^*$ be the $\F_p$-vector space in $E/\wp(E)$ generated by $[b]_E,[A]_E,[C]_E$ and $[\delta]_E$. Since 
\[
\begin{aligned}
\sigma_c(\delta)&= \delta +A,\\
\sigma_a(\delta)&=\delta+ C,\\
\sigma_a(A)&=A+b,\\
\sigma_c(C)&=C+b,
\end{aligned}
\]
we see that $W^*$ is in fact an $\F_p[G]$-module. Hence $M/F$ is a Galois extension by Artin-Schreier  theory.
\\
\\
{\bf Claim:} $\dim_{\F_p}(W^*)=4$. Hence $[L:F]=[L:E][E:F]=p^4p^2=p^6.$\\
{\it Proof of Claim:} From our hypothesis that $\dim_{\F_p}\langle [a]_F,[b]_F,[c]_F\rangle=3$, we see that $\langle [b]_E\rangle \simeq \F_p$. 

Clearly, $\langle[b]_E\rangle \subseteq (W^*)^G$.
From the relation 
\[
[\sigma_a(A)]_E= [A]_E+ [b]_E,
\]
we see that $[A]_E$  is not in $(W^*)^G$. Hence $\dim_{\F_p}\langle [b]_E,[A]_E\rangle=2$.

From the relation 
\[  [\sigma_c(C)]_E= [C]_E+ [b]_E,
\]
we see that $[C]_E$ is not in $(W^*)^{\sigma_c}$. But we have $\langle [b]_E,[A]_E\rangle\subseteq (W^*)^{\sigma_c}$. Hence
\[
\dim_{\F_p}\langle [b]_E,[A]_E,[C]_E\rangle =3.
\]

Observe that the element $(\sigma_a-1)(\sigma_c-1)$ annihilates the $\F_p[G]$-module $\langle [b]_E,[A]_E,[C]_E\rangle$, while 
\[
(\sigma_a-1)(\sigma_c-1)[\delta]_E= \sigma_a([A]_E)-[A]_E= [b]_E,
\]
we see that
\[
\dim_{\F_p}W^*=\dim_{\F_p}\langle [b]_E,[A]_E,[C]_E,[\delta]_E\rangle =4.
\]
\\
\\
Let $H^{a,b}=F(\theta_{a},\theta_{A},\theta_{b})$ and $H^{b,c}=F(\theta_{c},\theta_{C},\theta_{b})$. 
Let 
\[N:=H^{a,b}H^{b,c}=F(\theta_a,\theta_{c},\theta_{b},\theta_{A},\theta_{C})=E(\theta_{b},\theta_{A},\theta_{C}).\] 
Then $N/F$ is a Galois extension of degree $p^5$. 
This is because  ${\rm Gal}(N/E)$ is dual to the $\F_p[G]$-submodule $\langle  [b]_E,[A]_E,[C]_E \rangle$ via Artin-Schreier theory, 
and the proof of the claim above shows that $\dim_{\F_p}\langle [b]_E,[A]_E,[C]_E\rangle =3$. We have the following commutative diagram
\[
\xymatrix{
{\rm Gal}(N/F) \ar@{->>}[r] \ar@{->>}[d] & {\rm Gal}(H^{a,b}/F) \ar@{->>}[d]\\
{\rm Gal}(H^{b,c}/F) \ar@{->>}[r] & {\rm Gal}(F(\theta_{b})/F).
}
\]
So we have a homomorphism $\eta$ from ${\rm Gal}(N/F)$ to the pull-back ${\rm Gal}(H^{b,c}/F)\times_{{\rm Gal}(F(\theta_{b})/F)} {\rm Gal}(H^{a,b}/F)$:
\[
\eta\colon {\rm Gal}(N/F) \longrightarrow {\rm Gal}(H^{b,c}/F)\times_{{\rm Gal}(F(\theta_{b})/F)} {\rm Gal}(H^{a,b}/F),
\]
which make the obvious diagram commute.
We claim that $\eta$ is injective. Indeed, let $\sigma$ be an element in $\ker\eta$. Then $\sigma\mid_{H^{a,b}}=1$ in ${\rm Gal}(H^{a,b}/F)$, and $\sigma\mid_{H^{b,c}}=1$ in ${\rm Gal}(H^{b,c}/F)$. Since $N$ is the compositum of $H^{a,b}$ and $H^{b,c}$, this implies that $\sigma=1$, as desired. 

Since $|{\rm Gal}(H^{b,c}/F)\times_{{\rm Gal}(F(\theta_{b})/F)} {\rm Gal}(H^{a,b}/F)|=p^5=|{\rm Gal}(N/F)|$, we see that $\eta$ is actually an isomorphism.
As in the proof of Proposition~\ref{prop:Heisenberg char p}, we can choose an extension $\sigma_a\in {\rm Gal}(H^{a,b}/F)$ of $\sigma_a\in {\rm Gal}(F(\theta_{a},\theta_{b})/F)$ in such a way that 
\[
\sigma_a(\theta_{A})= \theta_{A}+\theta_{b}.
\]
Since the square commutative diagram above is a pull-back, we can choose an extension  $\sigma_a\in {\rm Gal}(N/F)$ of $\sigma_a\in {\rm Gal}(H^{a,b}/F)$ in such a way that
\[
 \sigma_a\mid_{H^{b,c}}=1.
\]
Now we can choose any extension $\sigma_a\in {\rm Gal}(M/F)$ of $\sigma_a\in{\rm Gal}(N/F)$. Then we have
\[
\sigma_a(\theta_{A})= \theta_{A}+\theta_{b} \; \text{ and }  \sigma_a\mid_{H^{b,c}}=1.
\]

Similarly, we can choose an extension  $\sigma_c\in {\rm Gal}(M/F)$ of $\sigma_c\in {\rm Gal}(F(\theta_b,\theta_c)/F)$ in such a way that
\[
\sigma_c(\theta_{C})= \theta_{C}+\theta_{b}, \text{ and } \sigma_c\mid_{H^{a,b}}=1.
\]

We define $\sigma_b\in {\rm Gal}(M/E)$ to be the element which is dual to $[b]_E$ via Artin-Schreier theory. In other words, we require that
\[
\sigma_b(\theta_{b})=1+ \theta_{b},
\]
and $\sigma_b$ acts trivially on $\theta_{A}$, $\theta_{C}$ and $\theta_{\delta}$. We consider $\sigma_b$ as an element in ${\rm Gal}(M/F)$, then it is clear that $\sigma_b$ is an extension of $\sigma_b\in {\rm Gal}(F(\theta_{a},\theta_{b},\theta_{c})/F)$.
Let $W={\rm Gal}(M/E)$, and let $H={\rm Gal}(M/F)$, then we have the following exact sequence
\[
1\to W\to H\to G\to 1.
\]
By Artin-Schreier theory, it follows that $W$ is dual to $W^*$, and hence $W\simeq (\Z/p\Z)^4$. In particular, we have $|H|=p^6$.

Recall that from \cite[Theorem 1]{BD}, we know that the group $\U_4(\F_p)$ has a presentation with generators $s_1,s_2,s_3$ subject to the  relations \eqref{eq:R} displayed in the proof of Theorem~\ref{thm:construction}.
Note that  $|\Gal(M/F)|=p^6$.
So in order to show that $\Gal(M/F)\simeq \U_4(\F_p)$, we shall show that $\sigma_a,\sigma_b$ and $\sigma_c$ generate $\Gal(M/F)$ and they satisfy these relations~\eqref{eq:R}.
 The proofs of the following claims are similar to the proofs of analogous claims in the proof of Theorem~\ref{thm:construction}. Therefore we shall omit them.
\\
\\
{\bf Claim:} The elements $\sigma_a$, $\sigma_b$ and $\sigma_c$ generate $\Gal(M/F)$.\\
{\bf Claim:} The order of $\sigma_a$ is $p$.\\
{\bf Claim:} The order of $\sigma_b$ is $p$.\\
{\bf Claim:} The order of $\sigma_c$ is $p$.\\
{\bf Claim:} $[\sigma_a,\sigma_c]=1$.\\
{\bf Claim:} $[\sigma_a,[\sigma_a,\sigma_b]]=[\sigma_b,[\sigma_a,\sigma_b]]=1$.\\
{\bf Claim:} $[\sigma_b,[\sigma_b,\sigma_c]]=[\sigma_c,[\sigma_b,\sigma_c]]=1$.\\
{\bf Claim:} $[[\sigma_a,\sigma_b],[\sigma_b,\sigma_c]]=1$.\\
\\
  An explicit isomorphism $\varphi\colon {\rm Gal}(M/F)\to \U_4(\F_p)$ may be defined as
\[
\sigma_a \mapsto \begin{bmatrix}
1& 1 & 0 & 0\\
0& 1 & 0 & 0\\
0& 0 & 1 & 0\\
0& 0 & 0 & 1
\end{bmatrix}, \; \;
\sigma_b\mapsto  \begin{bmatrix}
1& 0 & 0 & 0\\
0& 1 & 1 & 0\\
0& 0 & 1 & 0\\
0& 0 & 0 & 1
\end{bmatrix}, \;\;
 \sigma_c\mapsto \begin{bmatrix}
1& 0 & 0 & 0\\
0& 1 & 0 & 0\\
0& 0 & 1 & 1\\
0& 0 & 0 & 1
\end{bmatrix}.
\]

\end{proof}

\section{Triple Massey products}
Let $G$ be a profinite group and $p$ a prime number. We consider the finite field $\F_p$ as  a trivial discrete $G$-module. Let $\sC^\bullet=(C^\bullet(G,\F_p),\partial,\cup)$ be the differential graded algebra of inhomogeneous continuous cochains of $G$ with coefficients in $\F_p$ (see \cite[Ch.\ I, \S2]{NSW} and \cite[Section 3]{MT1}). For each $i=0,1,2,\ldots$, we write $H^i(G,\F_p)$ for the corresponding cohomology group. We denote by $Z^1(G,\F_p)$ the subgroup of $C^1(G,\F_p)$ consisting of all 1-cocycles. Because we use  trivial action on the coefficients $\F_p$, we have $Z^1(G,\F_p)=H^1(G,\F_p)={\rm Hom}(G,\F_p)$. Let $x,y,z$ be elements in $H^1(G,\F_p)$. Assume that 
\[
x\cup y=y\cup z=0\in H^2(G,\F_p).
\]
In this case we say that the triple Massey product $\langle x,y,z\rangle$ is defined. Then there exist cochains $a_{12}$ and $a_{23}$ in $C^1(G,\F_p)$ such that
\[
\partial a_{12}=x\cup y \; \text{ and } \partial a_{23}= y\cup z,
\]
in $C^2(G,\F_p)$. Then we say that $D:=\{x,y,z,a_{12},a_{23}\}$  is a {\it defining system} for the triple Massey product $\langle x, y, z\rangle$. Observe that 
\[
\begin{aligned}
\partial (x\cup a_{23}+ a_{12}\cup z) & = \partial x\cup a_{23}-x\cup \partial a_{23}+\partial a_{12}\cup z -a_{12}\cup \partial z\\
&=0 \cup a_{23}-x\cup (y\cup z) +(x\cup y)\cup z - a_{12}\cup 0\\
&=0 \in C^2(G,\F_p).
\end{aligned}
\]
Therefore $x\cup a_{23}+a_{12}\cup z$ is a 2-cocycle. We define the value $\langle x, y, z\rangle_D$ of the triple Massey product $\langle x, y, z\rangle$ with respect to the defining system $D$ to be the cohomology class $[x \cup a_{23}+ a_{12}\cup z]$ in $H^2(G,\F_p)$. The set of all values $\langle x, y, z\rangle_D$ when $D$ runs over the set of all defining systems, is called the triple Massey product $\langle x,y ,z \rangle \subseteq H^2(G,\F_p)$. Note that we always have
\[
\langle x,y ,z \rangle  = \langle x, y, z\rangle_D + x\cup H^1(G,\F_p) + H^1(G,\F_p)\cup z.
\]
We also have the following result.
 \begin{lem}
\label{lem:additivity} 
If the triple Massey products $\langle x,y,z\rangle $  and $\langle x,y^\prime,z\rangle$ are defined, then the triple Massey product $\langle x,y+y^\prime,z\rangle$ is defined, and
\[
\langle x,y+y^\prime,z\rangle=\langle x,y,z\rangle +\langle x,y^\prime,z\rangle.
\]
\end{lem}
\begin{proof} Let $\{x,y,z, a_{12}, a_{23}\}$ (respectively $\{x,y^\prime,z, a^\prime_{12}, a^\prime_{23}\}$) be a defining system for $\langle x,y,z\rangle$ (respectively $\langle x, y^\prime,z\rangle$). Then$ \{x,y+y^\prime,z, a_{12}+a^\prime_{12}, a_{23}+a^\prime_{23}\}$ is a defining system for $\langle x, y+y^\prime,z\rangle$. We also have
\[
\begin{aligned}
\langle x,y,z\rangle+ \langle x, y^\prime,z\rangle = &[x\cup a_{23}+a_{12}\cup z] +x\cup H^1(G,\F_p)+ H^1(G,\F_p) \cup z \\
&+ [x\cup a^\prime_{23}+a^\prime_{12}\cup z] +x\cup H^1(G,\F_p)+ H^1(G,\F_p) \cup z \\
= & [x\cup (a_{23}+a^\prime_{23})+ (a_{12}+a^\prime_{12})\cup z] + x\cup H^1(G,\F_p)+  H^1(G,\F_p)\cup z\\
=&\langle x,y+y^\prime,z\rangle,
\end{aligned}
\]
as desired.
\end{proof}
 
 The following lemma  is a special case of  a well-known fact (see \cite[Lemma 6.2.4 (ii)]{Fe}) but for the sake of convenience we provide its proof. 
 \begin{lem}
 \label{lem:scalar Massey}
 If the triple Massey product $\langle x,y,z\rangle $is defined, then for any $\lambda\in \F_p$ the triple Massey product $\langle x,\lambda y,z\rangle$ is defined, and
\[
\langle x,\lambda y,z\rangle\supseteq \lambda\langle x,y,z\rangle.
\]
  \end{lem}
  \begin{proof}
  Let $D=\{x,y,z, a_{12}, a_{23}\}$  be any defining system for $\langle x,y,z\rangle$. Clearly $D^\prime:=\{x,\lambda y,z, \lambda a_{12}, \lambda a_{23}\}$ is a defining system for $\langle x,\lambda y,z\rangle$, and 
  \[\lambda \langle x,y,z\rangle_D=\lambda[x\cup a_{23}+a_{12}\cup z]=[x\cup(\lambda a_{23})+(\lambda a_{12})\cup z]= \langle x,\lambda y,z\rangle_{D^\prime}.
  \]
  Therefore $\lambda\langle x,y,z\rangle\subseteq \langle x,\lambda y,z\rangle.$
  \end{proof}

A direct consequence of Theorems~\ref{thm:construction},~\ref{thm:construction char not p} and~\ref{thm:construction char p}, is the following result which roughly says that every "non-degenerate" triple Massey product vanishes whenever it is defined.
\begin{prop}
\label{prop:nondegenerate Massey}
Let $F$ be an arbitrary field. Let $\chi_1,\chi_2,\chi_3$ be elements in ${\rm Hom}(G_F,\F_p)$. We assume that $\chi_1,\chi_2,\chi_3$ are $\F_p$-linearly independent. If the triple Massey product $\langle \chi_1,\chi_2,\chi_3\rangle$ is defined, then it contains 0.
\end{prop}
\begin{proof}Let $L$ be the fixed field of $(F)^s$ under the kernel of the surjection $(\chi_1,\chi_2,\chi_3)\colon G_{F}\to (\F_p)^3$. Then Theorems~\ref{thm:construction}, \ref{thm:construction char not p} and~\ref{thm:construction char p} imply that $L/F$  can be embedded in a Galois $\U_4(\F_p)$-extension $M/F$. Moreover there exist $\sigma_1,\sigma_2,\sigma_3$ in ${\rm Gal}(M/F)$ such that they generate ${\rm Gal}(M/F)$, and 
  \[
  \begin{aligned}
  \chi_1(\sigma_1)=1,   \chi_1(\sigma_2)=0,   \chi_1(\sigma_3)=0;\\
  \chi_2(\sigma_1)=0,   \chi_2(\sigma_2)=1,   \chi_2(\sigma_3)=0;\\
  \chi_3(\sigma_1)=0,   \chi_3(\sigma_2)=0,   \chi_3(\sigma_3)=1.
  \end{aligned}
  \]
  (Note that for each $i=1,2,3$, $\chi_i$ is trivial on ${\rm Gal}(M/M_0)$, hence $\chi_i(\sigma_j)$ makes sense for every $j=1,2,3$.)
An explicit isomorphism $\varphi\colon {\rm Gal}(M/F)\to \U_4(\F_p)$ can be defined as
\[
\sigma_1 \mapsto \begin{bmatrix}
1& 1 & 0 & 0\\
0& 1 & 0 & 0\\
0& 0 & 1 & 0\\
0& 0 & 0 & 1
\end{bmatrix}, \; \;
\sigma_2\mapsto  \begin{bmatrix}
1& 0 & 0 & 0\\
0& 1 & 1 & 0\\
0& 0 & 1 & 0\\
0& 0 & 0 & 1
\end{bmatrix}, \;\;
 \sigma_3\mapsto \begin{bmatrix}
1& 0 & 0 & 0\\
0& 1 & 0 & 0\\
0& 0 & 1 & 1\\
0& 0 & 0 & 1
\end{bmatrix}.
\]
Let $\rho$ be the  composite homomorphism $\rho\colon\Gal_F\to \Gal(M/F)\stackrel{\varphi}{\simeq} \U_4(\F_p)$. Then one can check that
\[
\begin{aligned}
\rho_{12} &=\chi_1,\; \rho_{23}=\chi_2, \; \rho_{34} =\chi_3.
\end{aligned}
\]
(Since  all the maps $\rho,\chi_1,\chi_2,\chi_3$ factor through $\Gal(M/F)$, it is enough to check these equalities on elements  $\sigma_1,\sigma_2,\sigma_3$.) 
This implies that $\langle -\chi_1,-\chi_2,-\chi_3\rangle$ contains 0 by \cite[Theorem 2.4]{Dwy}. Hence $\langle \chi_1,\chi_2,\chi_3\rangle$ also contains 0.
\end{proof}

For the sake of completeness we include the following proposition, which together with Proposition~\ref{prop:nondegenerate Massey}, immediately yields a full new proof for a result which was first proved by E. Matzri \cite{Ma}. Matzri's result says  that defined triple Massey products vanish over all fields containing a primitive $p$-th root of unity. 
Alternative cohomological 
proofs for Matzri's result are in \cite{EMa2} and \cite{MT5}. Our new proof given in this section of the crucial "non-degenerate" part of this result (see Proposition~\ref{prop:nondegenerate Massey}),  which relies on explicit constructions of $\U_4 (\F_p)$-extensions, is a very natural proof because of Dwyer's result \cite[Theorem 2.4]{Dwy}. 
Observe that in \cite{MT5} we extended this result to all fields. 
\begin{prop}
\label{prop:degenerate Massey}
Assume that $\dim_{\F_p}\langle [a]_F,[b]_F,[c]_F\rangle\leq 2$. Then if the triple Massey product $\langle \chi_a,\chi_b,\chi_c\rangle$ is defined, then it contains 0.
\end{prop}
\begin{proof}
We can also assume that $a$, $b$ and $c$ are not in $(F^\times)^p$. The case that $p=2$, was treated in \cite{MT1}. So we shall assume that $p>2$.
\\
\\
{\bf Case 1:} Assume that $a$ and $c$ are linearly dependent modulo $(F^\times)^p$. This case is considered in \cite[Proof of Theorem 4.10]{MT5}. We include a proof here for the convenience of the reader. Let $\varphi=\{\varphi_{ab},\varphi_{bc}\}$ be a defining system for $\langle \chi_a,\chi_b,\chi_c\rangle$. We have
\[
\begin{aligned}
{\rm res}_{\ker \chi_a}(\langle \chi_a,\chi_b,\chi_c\rangle_\varphi)&= \res_{\ker\chi_a}(\chi_a\cup \varphi_{bc}+\varphi_{ab}\cup\chi_c )\\
&= \res_{\ker\chi_a}(\chi_a)\cup \res_{\ker\chi_a}(\varphi_{bc})+ \res_{\ker\chi_a}(\varphi_{ab})\cup\res_{\ker\chi_a}(\chi_c)\\
&= 0\cup \res_{\ker\chi_a}(\varphi_{bc})+ \res_{\ker\chi_a}(\varphi_{ab})\cup 0\\
& = 0.
\end{aligned}
\]
Then \cite[Chapter XIV, Proposition 2]{Se}, $\langle \chi_a,\chi_b,\chi_c\rangle_\varphi=\chi_a\cup\chi_x$ for some $x\in F^\times$. This implies that $\langle \chi_a,\chi_b,\chi_c\rangle$ contains 0.
 \\
\\
{\bf Case 2:} Assume that $a$ and $c$ are linearly independent. Then $[b]_F$ is in $\langle [a]_F,[c]_F\rangle$. Hence there exist $\lambda,\mu\in \F_p$ such that
\[
\chi_b = \lambda\chi_a +\mu\chi_c.
\]
Then we have
\[
\langle \chi_a,\chi_b,\chi_c\rangle = \langle \chi_a,\lambda\chi_a,\chi_c\rangle  +\langle \chi_a,\mu\chi_c,\chi_c\rangle \supseteq  \lambda \langle \chi_a,\chi_a,\chi_c\rangle  +\mu \langle \chi_a,\chi_c,\chi_c\rangle.
\]
(The equality follows from Lemma~\ref{lem:additivity} and the inequality follows from Lemma~\ref{lem:scalar Massey}.) By \cite[Theorem 5.9]{MT5} (see also \cite[Proof of Theorem 4.10, Case 2]{MT5}), $\langle \chi_a,\chi_a,\chi_c\rangle$ and $\langle \chi_a,\chi_c,\chi_c\rangle$ both contain 0. Hence $\langle \chi_a,\chi_b,\chi_c\rangle$ also contains 0.
\end{proof}

\begin{thm}
\label{thm:U4}
 Let $p$ be an arbitrary prime and  $F$ any field. Then the following statements are equivalent. 
 \begin{enumerate}
 \item There exist $\chi_1,\chi_2,\chi_3$ in ${\rm Hom}(G_F,\F_p)$ such that they are $\F_p$-linearly independent, and if ${\rm char} F\not=p$ then $\chi_1 \cup \chi_2 = \chi_2 \cup \chi_3 =0$.
  \item There exists a Galois extension $M/F$ such that ${\rm Gal}(M/F)\simeq \U_4(\F_p)$.
\end{enumerate}
Moreover, assume that (1) holds, and let $L$ be the fixed field of $(F)^s$ under the kernel of the surjection $(\chi_1,\chi_2,\chi_3)\colon G_{F}\to (\F_p)^3$. Then 
in (2) we can construct  $M/F$ explicitly such that  $L$ is embedded in $M$.

If $F$ contains a primitive $p$-th root of unity, then the two above conditions are also equivalent to the following condition.
\begin{enumerate}
\item[(3)] There exist $a,b,c\in F^\times$ such that $[F(\sqrt[p]{a},\sqrt[p]{b}, \sqrt[p]{c}):F]=p^3$ and $(a,b)=(b,c)=0$.
\end{enumerate}
If $F$ of characteristic $p$, then the two above conditions (1)-(2) are also equivalent to the following condition.
\begin{enumerate}
\item[(3')] There exist $a,b,c\in F^\times$ such that $[F(\theta_{a},\theta_{b}, \theta_{c}):F]=p^3$.
\end{enumerate}
 \end{thm}
\begin{proof} The implication that (1) implies (2), follows from Theorems~\ref{thm:construction}, \ref{thm:construction char not p} and~\ref{thm:construction char p}.
\\
\\
Now assume that $(2)$ holds.  Let $\rho$ be the composite $\rho\colon G_F\surj{\rm Gal}(M/F)\simeq \U_4(\F_p)$. Let $\chi_1:=\rho_{12}$, $\chi_2:=\rho_{23}$ and $\chi_3:=\rho_{34}$. Then $\chi_1, \chi_2,\chi_3$ are elements in ${\rm Hom}(G_F,\F_p)$, and $(\chi_1,\chi_2,\chi_3)\colon G_F\to (\F_p)^3$ is surjective. This implies that  $\chi_1,\chi_2,\chi_3$ are $\F_p$-linearly independent by \cite[Lemma 2.6]{MT4}. 

On the other hand, since $\rho$ is a group homomorphism, we see that 
\[
\chi_1\cup \chi_2 =\chi_2\cup\chi_3=0.
\]
Therefore $(1)$ holds.
\\
\\
Now we assume that $F$ contains a primitive $p$-th root of unity. Note that for any $a,b\in F^\times$, $\chi_a\cup\chi_b=0$, if and only if $(a,b)=0$ (see Subsection~\ref{subsec:norm residue}). 
Then (1) is equivalent to (3) by Kummer theory in conjunction with an observation that $[F(\sqrt[p]{a},\sqrt[p]{b}, \sqrt[p]{c}):F]=p^3$, if and only if  $\chi_a,\chi_b,\chi_c$ are $\F_p$-linearly independent.
\\
\\
Now we assume that $F$ of characteristic $p>0$. Then (1) is equivalent to (3') by Artin-Schreier theory in conjunction with an observation that $[F(\theta_{a},\theta_{b}, \theta_{c}):F]=p^3$, if and only if  $\chi_a,\chi_b,\chi_c$ are $\F_p$-linearly independent.
\end{proof}

\end{document}